\documentclass[12pt]{amsart} 

\usepackage[english]{babel}
\usepackage[latin1]{inputenc} 
\usepackage{amsmath}
\usepackage{amsthm}
\usepackage{amssymb}
\usepackage{tikz}

\usepackage{imakeidx}
\usepackage{appendix}
\usepackage{tikz-cd}
\usepackage{verbatim}
\usepackage{enumitem}
\usepackage{multicol}

\usepackage{hyperref}
\usepackage{cleveref}

\usepackage{mathtools}
\usepackage[all]{xy}

\numberwithin{equation}{section} 
\newtheorem{thm}[equation]{Theorem} 

\newtheorem{prop}[equation]{Proposition}
\newtheorem{lemma}[equation]{Lemma}
\newtheorem{cor}[equation]{Corollary}
\newtheorem{claim}[equation]{Claim}

\theoremstyle{definition}
\newtheorem{defin}[equation]{Definition}
\newtheorem{example}[equation]{Example}
\newtheorem{rmk}[equation]{Remark}
\newtheorem{rmks}[equation]{Remarks}
\newcommand{\Q}{\mathbb Q}

\newcommand{\Z}{\mathbb Z}
\newcommand{\Spec}{\operatorname{Spec}}
\newcommand{\Spin}{\operatorname{Spin}}

\newcommand{\Orth}{\operatorname{O}}
\newcommand{\Gr}{\operatorname{Gr}}
\newcommand{\Char}{\operatorname{char}}

\newcommand{\GL}{\operatorname{GL}}
\newcommand{\G}{\mathbb G}
\renewcommand{\P}{\mathbb P}
\newcommand{\C}{\mathbb C}

\newcommand{\bbR}{\mathbb R}

\newcommand{\A}{\mathbb A}

\newcommand{\ed}{\operatorname{ed}}

\newcommand{\trdeg}{\operatorname{trdeg}}

\newcommand{\mc}[1]{\mathcal{#1}}
\newcommand{\cl}{\overline}
\newcommand{\set}[1]{\{#1\}}
\newcommand{\on}[1]{\operatorname{#1}}
\newcommand{\ang}[1]{\left \langle{#1}\right \rangle}

\author{Zinovy Reichstein}
\address{Department of Mathematics\\
	University of British Columbia\\
	Vancouver, BC V6T 1Z2\\Canada}
\email{reichst@math.ubc.ca}
\thanks{Zinovy Reichstein was partially supported by
	National Sciences and Engineering Research Council of
	Canada Discovery grant 253424-2017.}

\author{Federico Scavia}
\address{Department of Mathematics\\
	University of California\\
	Los Angeles, CA 90095-1555 \\ USA}
\email{scavia@math.ucla.edu} 
\thanks{Federico Scavia was partially supported by a graduate fellowship from the University of British Columbia.}



\title[Essential dimension and specialization II]{The behavior of essential dimension under specialization II}
\begin{document}
	
	\begin{abstract} Let $G$ be a linear algebraic group over a field. We show that, under mild assumptions, in a family of primitive generically free $G$-varieties over a base variety $B$ the essential dimension of the geometric fibers may drop on a countable union of Zariski closed subsets of $B$
	and stays constant away from this countable union. 
	We give several applications of this result.
	\end{abstract} 
	
	\maketitle
	
	\section{Introduction}
	
	Let $X$ be a complex algebraic variety (that is, a separated reduced $\mathbb{C}$-scheme of finite type) equipped with a faithful action of a finite group $G$. We will refer to $X$ as a $G$-variety. Assume that the $G$-action on $X$ is primitive, that is, $G$ transitively permutes the irreducible components of $X$. In this paper we will be interested in the essential dimension $\ed_{\C}(X; G)$ and how it behaves in families. Essential dimension is an integer-valued birational invariant of the $G$-variety $X$; its definition can be found in Section~\ref{sect.prelim1}. When the group $G$ is clear from the context, we will simply write $\ed_{\C}(X)$ for $\ed_{\C}(X; G)$.
	
	To date the study of essential dimension has been primarily concerned with understanding versal $G$-varietes 
	(once again, see Section~\ref{sect.prelim1} for the definition). A complete versal $G$-variety $X$ has the following special property: 
	$X$ has an $A$-fixed rational point for every abelian subgroup $A \subset G$; see~\cite[Corollary 3.21]{merkurjev2009essential}.   
    At the other extreme are complete $G$-varieties $X$, where the action of $G$ is free, i.e., no non-trivial element has a fixed point.  
    Existing methods for proving lower bounds on $\ed_{\C}(X)$ usually fail here; 
    we are aware of only a small number of interesting examples of finite group actions, where $\ed_{\C}(X)$ has been computed in this setting.  
	
	One such family concerns the action of $G = (\Z/ p \Z)^n$ on the product of elliptic curves $X =E_1 \times \ldots \times E_n$ 
	over $\C$.	Here $p$ is a prime; the generator of the $i$th copy of $\Z/ p \Z$ acts on $E_i$ 
	via translation by a point $x_i \in E_i(\C)$ of order $p$, and trivially on $E_j$ for $j \neq i$.
	J.-L.~Colliot-Th\'el\`ene and O.~Gabber~\cite[Appendice]{colliot2002exposant} showed that for a very general 
	choice of the elliptic curves $E_i$, a certain degree $n$ cohomological invariant of $G$ does not vanish on 
	$\C(X)^G$. This implies that $\ed_{\C} (X) = n = \ed_{\C}(G)$.	
	Additional examples can be found in recent work of B.~Farb, M.~Kisin and J.~Wolfson~\cite{fkw, fkw2} and
	N. Fakhruddin and R.~Saini~\cite{fakhruddin2021finite}. Some arise as 
	congruence covers of Shimura varieties, others from actions of subgroups of $X[p] \simeq (\Z/ p \Z)^{2n}$
	on a complex abelian variety $X$ (not necessarily a product of elliptic curves).

	In the present paper we will work in a more general setting, where $G$ is a linear algebraic group (not necessarily finite) over an algebraically closed field $k$. Essential dimension and versality make sense in this more general setting, 
	provided that we require our $G$-actions to be generically free and not just faithful; see Section~\ref{sect.prelim1}.	
	
	If $\dim(G)>0$, by Borel's Fixed Point Theorem $G$ cannot act freely on a complete variety. Nevertheless, the notion of a free action of a finite group
	on a projective variety can be generalized to the case of an arbitrary linear algebraic group $G$ as follows: 
    a generically free primitive $G$-variety $X$ is said to be \emph{strongly unramified} 
    if $X$ is birationally $G$-equivariantly isomorphic to the total space $X'$ of a $G$-torsor $X'\to P$ over 
    some smooth projective irreducible $k$-variety $P$. 

Our first main result is the following.
	
	\begin{thm} \label{thm.free}
		Let $G$ be a linear algebraic group over an algebraically closed field
		$k$ of good characteristic (see Definition~\ref{assume}) and of infinite transcendence degree over its prime field, and let $X$ be a generically free primitive $G$-variety. 
		Then there exist a strongly unramified $G$-variety $Y$ such that $\dim(Y) = \dim(X)$ and $\ed_k(Y) = \ed_k(X)$.
	\end{thm}
	Applying \Cref{thm.free} to a versal $G$-variety $X$, we obtain a strongly unramified $G$-variety $Y$ of maximal essential dimension, i.e. such that $\ed_k(Y)=\ed_k(G)$. When $G$ is finite, $Y$ is itself smooth and projective.
    Thus by starting with an incompressible $G$-variety $X$, we obtain examples   
    analogous to those of Colliot-Th\'el\`ene--Gabber, Farb--Kisin--Wolfson and Fakhruddin--Saini for an arbitrary finite group $G$. 
    Note however, that Farb-Kisin-Wolfson and Fakhruddin-Saini produce examples over $k = \overline{\Q}$, whereas Theorem~\ref{thm.free} 
    requires $k$ to be of infinite transcendence degree over the prime field.
	
	Our proof of~\Cref{thm.free} will rely on Theorems~\ref{verygen} and~\ref{complete-int} below, which are of independent interest.
	
	\begin{thm}\label{verygen}
		Let $G$ be a linear algebraic group over a field $k$ of good characteristic. Let $B$ be a noetherian $k$-scheme, $f \colon \mc{X}\to B$ be a flat separated $G$-equivariant morphism of finite type such that $G$ acts trivially on $B$ and the geometric fibers of $f$ are generically free and primitive $G$-varieties (in particular, reduced). Then for any fixed integer $n \geqslant 0$ the subset of $b\in B$ such that $\ed_{k(\cl{b})}(\mc{X}_{\cl{b}};G_{k(\cl{b})})\leqslant n$ for every (equivalently, some) geometric point $\cl{b}$ above $b$ is a countable union of closed subsets of $B$.
		
		Furthermore, assume that $k$ is algebraically closed of infinite transcendence degree over its prime field. (In particular, these 
		conditions are satisfied if $k$ is algebraically closed and uncountable.) Let $m\geqslant 0$ be the maximum of $\ed_{k(\cl{b})}(\mc{X}_{\cl{b}};G_{k(\cl{b})})$, where $\cl{b}$ ranges over all geometric points of $B$. Then the set of those $b\in B(k)$ such that $\ed_k(\mc{X}_b;G)=m$ is Zariski dense in $B$.
	\end{thm}
	
	Informally Theorem~\ref{verygen} can be restated as follows: in a family of $G$-varieties $\mc{X}\to B$, the essential dimension of the geometric fibers drops on a countable union of Zariski closed subsets of $B$, and stays constant away from this countable union. Several remarks concerning Theorem~\ref{verygen} are in order.
	
	\begin{rmks} \label{rmks.verygen}
		\begin{enumerate}
			\item The assumption that the $G$-action on every geometric fiber $\mc{X}_{b}$ of $f$ is generically free 
			and primitive ensures that $\ed_{k(b)}(\mc{X}_{b})$ is well defined.
			
			\item  The countable union in the statement of \Cref{verygen}(a) cannot be replaced by a finite union in general; see Example~\ref{ex.general}. 
			
			\item The assumption that $f$ is flat is necessary; see~\Cref{flatness-necessary}. On the other hand, this assumption is rather mild. For example, when $\mc{X}$ and $B$ are smooth $k$-varieties, by ``miracle flatness", $f$ is flat if and only if all its fibers have the same dimension; see \cite[Theorem 23.1]{matsumura1989commutative}. In the applications, one is usually interested in showing that the maximal value of $\ed_k(\mc{X}_b)$ is attained at a very general point $b\in B(k)$. This can be done under a weaker flatness assumption 
			on $f$; see \Cref{verygen-sing}. 
			
			\item If $k$ is not algebraically closed, then the $k$-points $b \in B(k)$ such that $\ed_k(\mc{X}_b) \leqslant n$ 
			do not necessarily lie on a countable union of closed subvarieties of $B$; see~\Cref{alg.closure-necessary}. In other words, \Cref{verygen} 
			may fail if we consider fibers of arbitrary closed points instead of just geometric fibers.
			
		\end{enumerate}
	\end{rmks}

	Our proof of \Cref{verygen} proceeds as follows. First we choose a subfield $k_0\subset k$ finitely generated over $\Q$, such that $G=G_0\times_{\Spec(k_0)} \Spec(k)$, $f=f_0\times_{\Spec(k_0)} \Spec(k)$, and the assumptions of \Cref{verygen} hold for $k_0$, $G_0$ and $f_0:\mc{X}_0\to B_0$.
	Then using arguments inspired by Gabber's appendix~\cite{colliot2002exposant} we reduce \Cref{verygen} to 
   the Specialization Property (\Cref{ed-dvr}) and the Rigidity Property (\Cref{rigidity}).

	Note that the Rigidity Property may fail if $k$ is not algebraically closed. This is the reason why in~\Cref{verygen} we 
	only consider the geometric fibers; see Remark \ref{rmks.verygen}(4).	
		
	\begin{thm}\label{complete-int}
		Let $k$ be an infinite field, $G$ be a finite group, and let $X_0$ be an equidimensional generically free $G$-variety of dimension 
		$e\geqslant 1$ (not necessarily primitive). 
		Then there exist a smooth irreducible $k$-variety $B$, a smooth 
		irreducible $G$-variety $\mc{X}$ and a smooth $G$-equivariant morphism $f \colon \mc{X}\to B$ 
		of constant relative dimension $e$ defined over $k$ such that:
		\begin{enumerate}[label=(\roman*)]
			\item $G$ acts trivially on $B$ and freely on $\mc{X}$,
			\item there exist $b_0 \in B(k)$ such that $X_0$ is $G$-equivariantly birationally isomorphic to a union of irreducible components of $\mc{X}_{b_0}$, 
			\item there exists a dense open subscheme $U \subset B$ such that for every $b\in U$ the fiber $\mc{X}_b$ is smooth, 
			projective and geometrically irreducible.
		\end{enumerate}
		
			In particular, for any a geometric point $b$ of $U$, the $G$-action on the fiber $\mc{X}_b$ is strongly unramified.
	\end{thm}
	
Our proof of \Cref{complete-int} can be found in Section~\ref{sect.family}. It was motivated by J.-P.~Serre's construction of a smooth projective $n$-dimensional complete intersection with a free $G$-action, for an arbitrary finite group $G$ and an arbitrary positive integer $n$; 
see~\cite[Proposition 15]{serre58charp}. \Cref{thm.free} is then deduced from		
Theorems~\ref{verygen} and~\ref{complete-int} in Section~\ref{sect.thm.free}.

This paper is a sequel to~\cite{specialization1}. The main result of~\cite{specialization1} is used in the proof of~\Cref{ed-dvr} 
(the specialization property of essential dimension). Other than that, this paper can be read independently of~\cite{specialization1}.

	\section{Notation and preliminaries}
	\label{sect.prelim}
	
	\subsection*{Group actions and essential dimension}
	\label{sect.prelim1}
		Let $k$ be a field, $\cl{k}$ be an algebraic closure of $k$, $G$ be a linear algebraic group over $k$, 
and $X$ be a $G$-variety, i.e., a separated reduced $k$-scheme of finite type endowed with a $G$-action over $k$. We will say 
that the $G$-variety $X$ is primitive if $X\neq \emptyset$, $G(\cl{k})$ transitively permutes the irreducible components of $X_{\cl{k}}:=X\times_k\cl{k}$, and 
generically free if there exists a dense open subscheme $U\subset X$ such that for every $u\in U$ the scheme-theoretic stabilizer 
$G_u$ of $u$ is trivial. 

 By a $G$-compression of $X$ we will mean a dominant $G$-equivariant rational map $X \dashrightarrow Y$, where the $G$-action on $Y$ is again
generically free and primitive. The essential dimension of $X$, denoted by $\ed_k(X; G)$, or $\ed_k(X)$ if $G$ is clear from the context, is defined as the minimal value of $\dim(Y)$, 
where the minimum is taken over all $G$-compressions $X \dashrightarrow Y$. The essential dimension $\ed_k(G)$ of the group $G$ 
is defined as the supremum of $\ed_k(X)$, where $X$ ranges over all faithful primitive $G$-varieties.
	
A $G$-variety $X$ is called weakly versal if every generically free
primitive $G$-variety $T$ admits a $G$-equivariant rational map $T\dashrightarrow U$. We will say that $X$ is versal if every
dense open $G$-invariant subvariety $U \subset X$ is weakly versal.

\subsection*{Good characteristic}

\begin{defin} \label{assume} Let $G$ be a linear algebraic group defined over a field $k$. We will say that $G$ is in {\em good characteristic} 
if
	\begin{itemize}
	    \item either $\on{char}(k)=0$, or
	    \item $\on{char}(k)=p>0$, 
	    $G^{\circ}$ is smooth reductive and there exists a finite subgroup $S\subset G(\cl{k})$ of order prime to $p$ such that the induced map $H^1(K,S)\to H^1(K,G)$ is surjective for every field extension $K/\cl{k}$, or
	    \item $G$ is a finite discrete group, and if $\on{char}(k)=p>0$ then the only normal $p$-subgroup of $G$ is the trivial subgroup (that is, $G$ is weakly tame in the sense of \cite{brosnan2018essential}). 
	\end{itemize} 
	\end{defin}

Here are two large families of examples in prime characteristic.
	
\begin{example} \label{ex.good-char} Suppose $G$ is a smooth group over a field $k$ of characteristic $p > 0$. Assume that the connected component $G^0$ of $G$
is reductive. Let $T$ be a maximal torus in $G^0$, $r = \dim(T) \geqslant 0$, and $W = N_G(T)/T$ be the Weyl group. If
 
 \smallskip
(a) $G$ is split and defined over $\Spec(\mathbb Z)$ and $p$ does not divide $2^r |W|$, or

\smallskip
(b) $G$ is connected and $p$ does not divide $|W|$,

\smallskip \noindent
then $G$ is in good characteristic. For a proof of (a), see~\cite[Proposition 5.1]{specialization1}. For a proof of (b), 
see~\cite[Theorem 1.1(c)]{chernousov2006resolving} and~\cite[Remark 4.1]{chernousov2008reduction}.
 \end{example}
 
The following example shows that conditions (a) and (b) above can sometimes be relaxed. 

\begin{example}
The split orthogonal group $\on{O}_n$, special orthogonal group $\on{SO}_n$ and the spin group $\on{Spin}_n$ over a field $k$
are in good characteristic as long as $\on{char}(k) \neq 2$. Indeed, let $S$ be the group of diagonal $n \times n$ matrices of the form $\on{diag}(\epsilon_1, \ldots, \epsilon_n)$,
where each $\epsilon_i = \pm 1$, $S_0 = S \cap \on{SL}_n$, and $\tilde{S}$ be the preimage of $S_0$ under the natural map $\on{Spin}_n \to \on{SO}_n$.
Then $|S| = |\tilde{S}| = 2^n$, $|S_0| = 2^{n-1}$, and if $\on{char}(k) \neq 2$, then the natural maps 
\[ H^1(K, S) \to H^1(K, \on{O}_n), \; \; H^1(K, S_0) \to H^1(K, \on{SO}_n), \; \; \text{and} \; \; H^1(K, \tilde{S}) \to H^1(K, \Spin_n) \]
are all surjective. The surjectivity of the first two maps follows from the fact that every quadratic form over a field of characteristic $\neq 2$ can be diagonalized. The surjectivity of the third map is proved in \cite[Lemma 13.2]{brosnan2007-arxiv}.
\end{example}

	\section{Specialization Property}
	\label{sect.specialization}
	The purpose of this section is to prove the following specialization property of essential dimension.
	
	\begin{prop}\label{ed-dvr}
	Let $k$ be a field, $R$ be a discrete valuation ring containing $k$ and with residue field $k$, and $l$ be the fraction field of $R$. Let $G$ be a linear algebraic group over a field $k$ of good characteristic. Let $X$ be a flat separated $R$-scheme of finite type endowed with a $G$-action over $R$, whose fibers are generically free and primitive $G$-varieties. 
	Then $\ed_{\cl{l}}(X_{\cl{l}})\geqslant \ed_{\cl{k}}(X_{\cl{k}})$.
	\end{prop}
	
	 Our proof will be based on reduction to the case where $X$ is a $G$-torsor over $\Spec(R)$. In the latter case the inequality
	 $\ed_{\cl{l}}(X_{\cl{l}})\geqslant \ed_{\cl{k}}(X_{\cl{k}})$ of~\Cref{ed-dvr} is established in~\cite[Theorem 6.4]{specialization1}.

\begin{proof}[Proof of \Cref{ed-dvr}]
Let $k$ be the residue field of $R$, and $l$ be the fraction field of $R$. By assumption, $X_k$ (resp. $X_l$) is a  primitive generically free $G_k$-variety (resp. $G_l$-variety). We fix algebraic closures $\cl{k}$ and $\cl{l}$ of $k$ and $l$, respectively. Our proof will be in several steps.

\begin{claim}\label{fiberdimension}
There exists an integer $d\geqslant 0$ such that the irreducible components of $X_{\cl{k}}$ and of $X_{\cl{l}}$ are all of dimension $d$.
\end{claim}

For any finite field extensions $k'\supset k$ and $l'\supset l$, there exists a discrete valuation ring $R'\supset R$, finite and free over $R$, such that the residue field of $R'$ contains $k'$ and the fraction field of $R'$ contains $l'$; see \cite[I.4, Proposition 9 and Remark]{serre1979local} and \cite[I.6, Proposition 15]{serre1979local}. Thus, extending $R$ if necessary, we may assume that the irreducible components of $X_k$ (resp. $X_l$) are geometrically irreducible and transitively permuted by $G(k)$ (resp. $G(l)$).

After this reduction, the problem becomes to find $d\geqslant 0$ such that the irreducible components of $X_{k}$ and of $X_{l}$ are all of dimension $d$. Since $G$ acts transitively on the irreducible components of the fibers, it suffices to exhibit one irreducible component of $X_k$ and one irreducible component of $X_l$ of the same dimension.

Since $X$ is $R$-flat, by \cite[Lemma 4.3.7]{liu2002algebraic} every irreducible component of $X$ dominates $\Spec R$. In other words, the open subscheme $X_l\subset X$ is dense. Therefore each irreducible component of $X$ is the closure of an irreducible component of $X_l$. Thus, since $X_k\neq\emptyset$, there exists an irreducible component $X'\subset X$ such that $X'_k$ contains some irreducible component of $X_k$ and such that $X'_l$ is an irreducible component of $X_l$.

The composition $X'\hookrightarrow X\to \Spec R$ is surjective, hence \cite[0B2J]{stacks-project} implies that every irreducible component of $X'_k$ has dimension $\on{dim}(X'_l)$. Since $X'_k$ contains an irreducible component of $X_k$, this completes the proof of \Cref{fiberdimension}.

    \begin{claim}\label{lem.free}
    There exists a $G$-invariant $R$-fiberwise dense open subscheme $U\subset X$ such that $G$ acts freely on $U$.
    \end{claim}
    
    Since $G_l$ acts generically freely on $X_l$, there exists a closed nowhere dense $G_l$-invariant subscheme $Z\subset X_l$ such that $G_l$ acts freely on $X_l\setminus Z$. Let $W\subset Z$ be an irreducible component, and let $\cl{W}$ be the closure of $W$ in $X$. By \cite[Tag 0B2J]{stacks-project}, either $(\cl{W})_{k}$ is empty, or \[\dim ((\cl{W})_{k})=\dim(W)\leqslant \dim(X_{k})-1.\] It now follows from \Cref{fiberdimension} that $\cl{W}$ does not contain any irreducible component of $X_{k}$. Therefore, the closure $\cl{Z}$ of $Z$ does not contain any irreducible component of $X_{k}$. 
		
	Since $G_k$ acts generically freely on $X_{k}$, there exists a closed nowhere dense $G_{k}$-invariant subscheme $Z'\subset X_{k}$ such that $G_{k}$ acts freely on $X_{k}\setminus Z'$. It follows that $U:=X\setminus (\cl{Z}\cup Z')$ is a fiber-wise dense $G$-invariant open subscheme of $X$, such that $G$ acts freely on $U_l$ and $U_k$. To prove \Cref{lem.free}, it remains to show that $G$ acts freely on $U$, i.e., that the stabilizer $U$-group scheme 
	\[\mc{G}:=U\times_{(U\times_R U)} (G\times_RU) \] 
	is trivial. Here the fibered product is taken over the diagonal morphism $U\to U\times_RU$ and the action morphism $G\times_RU\to U\times_R U$. Since $G$ acts freely on $U_l$ and $U_k$, the $U_l$-group scheme $\mc{G}_l$ and the $U_k$-group scheme $\mc{G}_k$ are both trivial. Hence so is $\mc{G}$, as desired. This proves \Cref{lem.free}.
	
    \begin{claim}\label{torsor-of-schemes}
    For the purpose of proving \Cref{ed-dvr}, we may assume that $X$ is the total space of a $G$-torsor $X\to Y$, where $Y$ is a separated $R$-scheme of finite type.
    \end{claim}
   
    After replacing $X$ by the open $R$-fiberwise dense subscheme $U$ constructed in \Cref{lem.free}, we may assume that $G$ acts freely
    on $U$. Let $X/G\to \Spec R$ denote the fppf-quotient of $X\to \Spec R$ by the $G$-action. By a theorem of M. Artin~\cite[Th\'eor\`eme 3.1.1]{sivaramakrishna1973schemas}, $X/G$ is represented by an algebraic space of finite type $Y$ over $R$. (Equivalently, the quotient stack $[X/G]$ has trivial stabilizers, hence it is represented by an algebraic space.) Since $G$ is smooth, the projection $X\to X/G$ is 
    an \'etale torsor. 
    
    We claim that $X$ is flat over $R$. If $Y$ were a scheme, this would follow directly from~\cite[Tag 02JZ]{stacks-project}, since $X$ is flat over both $R$ (this is one of the assumptions of~\Cref{ed-dvr}) and $Y$ (because $X \to Y$ is a $G$-torsor). If general we use the fact that $R$-flatness is an \'etale local property: to prove that a representable morphism of algebraic spaces is flat, it suffices to do so locally after an \'etale base change. This way we reduce the claim to the case, where $Y$ is a scheme and~\cite[Tag 02JZ]{stacks-project} applies.
	
	Let $Y'\to Y$ be surjective \'etale morphism and $X':=X\times_YY'$. For every $y\in Y$ and every geometric point $\cl{y}$ of $Y$ lying above $y$, $\cl{y}$ factors through some $y'\in Y'$ and we have a $G_{k(y')}$-equivariant isomorphism $(X_y)_{k(y')}\simeq (X')_{y'}$. It follows that passage to an \'etale cover of $Y$ does not alter the essential dimension of the geometric fibers of $X\to Y$. Therefore, we may assume that $Y$ is a scheme, that is, $X\to Y$ is an \'etale $G$-torsor in the category of schemes. This proves \Cref{torsor-of-schemes}.
	
		\begin{claim}\label{torsor-over-dvr}
	For the purpose of proving \Cref{ed-dvr}, we may assume that:
	\begin{itemize}
	    \item $X$ is the total space of a $G$-torsor $X\to \Spec A$, where $A$ is a discrete valuation ring containing $R$ and the inclusion $R\subset A$ is a local homomorphism, and
	    \item $\ed_{l}(X_{l}) = \ed_{\cl{l}}(X_{\cl{l}})$.
	\end{itemize}
	\end{claim}
	By \Cref{torsor-of-schemes}, we may assume that $X$ is the total space of a $G$-torsor $X \to Y$, where $Y$ is a separated 
	$R$-scheme of finite type. Since $G$ acts primitively on $X_k$ and $X_l$, the varieties $Y_k$ and $Y_l$ are geometrically irreducible. 
	By \cite[Lemma 054F]{stacks-project}, there exists a morphism $\Spec A\to Y$, where $A$ is a discrete valuation ring whose generic point $\eta$ maps to the generic point $y$ of $Y$ so that the induced inclusion $k(y)\subset k(\eta)$ is an equality, and whose closed point $s$ maps to the generic point $y'\in Y_k$. We have
	\[\trdeg_kk(y')=\dim(Y_k)=\dim(Y_l)-1=\trdeg_kk(y)-1=\trdeg_kk(\eta)-1=\trdeg_kk(s),\]
	hence the field extension $k(s)/k(y')$ is algebraic.
	We thus obtain the following Cartesian diagram. 
	\[ \xymatrix{  X_A \ar@{->}[d]  \ar@{->}[r] &  X \ar@{->}[d]     \\          
			\Spec(A) \ar@{->}[r]     &            Y \ar@{->}[d] \\
			                         &     \Spec(R), } \]
	where $X_A = X \times_{\Spec(R)} \Spec(A)$. By construction, the morphism $\Spec A\to \Spec R$ sends the closed point of $\Spec A$ to the closed point of $\Spec R$, and so it is local. As in \Cref{torsor-of-schemes}, replacing the $G$-torsor $X \to Y$ by the $G$-torsor $X_A \to \Spec A$ 
	does not alter the essential dimension of the geometric fibers over the generic and the closed points of $R$. This proves the first assertion of \Cref{torsor-over-dvr}.
	
	Since every $G_{\cl{l}}$-equivariant compression of $X_{\cl{l}}$ over $\cl{l}$ is defined over some finite extension of $l$, there is a finite subextension $l\subset l'\subset \cl{l}$ such that
	$\ed_{l'}(X_{l'}) = \ed_{\cl{l}}(X_{\cl{l}})$.
	Let $R'\supset R$ be a discrete valuation ring with fraction field $l'$, and let $k'\supset k$ be the residue field of $R'$. The $G_R$-torsor $X$ over $R$ lifts to a  $G_{R'}$-torsor on $X_{R'}$, which is $R'$-fiberwise generically free and primitive. Since  $\ed_{k'}(X_{k'})\geqslant \ed_{\cl{k}}(X_{\cl{k}})$, we are allowed to replace $R$ by $R'$. This completes the proof of \Cref{torsor-over-dvr}.
	
		\smallskip
    We are now ready to complete the proof of \Cref{ed-dvr}. We may place ourselves in the setting of \Cref{torsor-over-dvr}. Since $k$ is algebraically closed and $G$ is in good characteristic, the assumptions of \cite[Theorem 6.4]{specialization1} are satisfied, hence $\ed_{l}(X_{l})\geqslant \ed_{k}(X_{k})$. Therefore  
    \[\ed_{\cl{l}}(X_{\cl{l}})=\ed_{l}(X_{l})\geqslant  \ed_{k}(X_{k})\geqslant \ed_{\cl{k}}(X_{\cl{k}}).\qedhere\]
\end{proof}

	\section{Proof of Theorem \ref{verygen}}
	\label{sect.rigidity}
	
	\begin{lemma}\label{mumford}
	Let $k/k_0$ be a field extension of infinite transcendence degree such that $k$ is algebraically closed. Let $B_0$ be an irreducible $k_0$-variety and $B:=B_0\times_{k_0}k$. Then the set of $k$-rational points of $B$ mapping to the generic point of $B_0$ is dense in $B$.
	\end{lemma}
	
	\begin{proof}
	  Let $U$ be a non-empty open $k$-subscheme of $B$. It suffices to prove the following
	  
	  \smallskip
	  Claim: $U$ has a $k$-point which maps to the generic point of $B_0$.
	  
	  \smallskip
	  To prove this claim, note that the open embedding $U\hookrightarrow B$ is defined over some intermediate subfield $k_0 \subset k_1 \subset k$ such that the extension $k_1/k_0$ is finitely generated. In other words,  
	  $U\hookrightarrow B$ is obtained by base change from an affine open embedding $U_1\hookrightarrow B_0\times_{k_0}k_1$ defined over $k_1$. In particular, the morphism $U_1\to B_0$ is dominant. Let $\eta_1:\Spec K_1\to U_1$ be the generic point of $U_1$. 
	  
	  Now consider a subfield $k_1\subset L_1\subset K_1$ such that $L_1/k_1$ is purely transcendental of finite transcendence degree and $K_1/L_1$ is finite. Since $k/k_1$ has infinite transcendence degree, there exists a field embedding $\iota:L_1\hookrightarrow k$ compatible with the $k_1$-algebra structures of $L_1$ and $k$. Since $k$ is algebraically closed and  $K_1/L_1$ is finite, we may extend $\iota$ to a field embedding $K_1\hookrightarrow k$, again compatible with the $k_1$-algebra structures of $K_1$ and $k$. This gives rise to a scheme morphism \[u_1:\Spec k\to \Spec K_1\xrightarrow{\eta_1} U_1.\] Since $U=U_1\times_{k_1}k$, $u_1$ uniquely lifts to
	  a $k$-point $u$ of $U$ mapping to the generic point of $U_1$. Since the morphism $U_1\to B_0$ is dominant, the $k$-point $u$ maps to the generic point of $B_0$. This completes the proof of the Claim and thus of \Cref{mumford}.
	\end{proof}

	We will make use of the following ``rigidity property" of essential dimension. For a proof, see~\cite[Lemma 2.2]{specialization1}.
	
	\begin{lemma}\label{rigidity} Let $k$ be an algebraically closed field, $G$ be a $k$-group, and $X$ be a generically free primitive $G$-variety defined over $k$. Then	$\ed_k(X) = \ed_l(X_l)$ for any field extension $l/k$. \qed
	\end{lemma}
		
     Not let $f \colon \mc{X}\to B$ be as in \Cref{verygen}. For every integer $n$, we set 
	\[\Phi_f(n):=\set{b\in B \, | \, \ed_{k(\cl{b})}(\mc{X}_{\cl{b}})\leqslant n \text{ for some geometric point $\cl{b}$ with image $b$}}.\]
	
	\begin{lemma}\label{phi}
		(a) A point $b\in B$ belongs to $\Phi_f(n)$ if and only $\ed_{k(\cl{b})}(\mc{X}_{\cl{b}}) \leqslant n$ for every geometric point $\cl{b}$ with image $b$.
		
		\smallskip
		(b) Let $\pi:B'\to B$ be a morphism of schemes, and let $f':X\times_BB'\to B'$ be the base of change of $f$ along $\pi$. Then $\Phi_{f'}(n)=\pi^{-1}(\Phi_f(n))$.
	\end{lemma}
	
	\begin{proof}
		(a) Let $\cl{b}_1$ and $\cl{b}_2$ be two geometric points of $B$ with image $b$. By \cite[Exercise 3.1.10(b)]{liu2002algebraic}, the ring $A:=k(\cl{b}_1)\otimes_{k(b)}k(\cl{b}_2)$ is not zero. If $\mathfrak{m}$ is a maximal ideal of $A$, the quotient $A/\mathfrak{m}$ is a field containing $k(\cl{b}_1)$ and $k(\cl{b}_2)$. By considering an algebraic closure of $A/\mathfrak{m}$, we are thus reduced to the case when there is a field homomorphism $k(\cl{b}_1)\hookrightarrow k(\cl{b}_2)$. We may thus assume that $k(\cl{b}_1)\subset k(\cl{b}_2)$. In this case, (a) follows from~\Cref{rigidity}.
		
		(b) Let $b'\in B'$ and $b\in B$ be such that $\pi(b')=b$. Let $\cl{b}'$ be a geometric point of $B$ with image $b'$, so that $\cl{b}:=\pi\circ \cl{b}'$ is geometric point of $B$ with image $b$. Then there is a natural isomorphism $\mc{X}_{\cl{b}'}\simeq \mc{X}_{\cl{b}}\times_{k(\cl{b})}k(\cl{b}')$ of $G_{\cl{b}'}$-varieties, and 
		\[\ed_{k(\cl{b})}(\mc{X}_{\cl{b}})=\ed_{k(\cl{b}')}(\mc{X}_{\cl{b}}\times_{k(\cl{b})}k(\cl{b}'))= \ed_{k(\cl{b}')}(\mc{X}_{\cl{b}'})\] 
		by \Cref{rigidity}(c). In particular, $b\in \Phi_f(n)$ if and only if $b'\in \Phi_{f'}(n)$, as desired.
	\end{proof}

	\begin{proof}[Proof of \Cref{verygen}]
		We must show that $\Phi_f(n)\subset B$ is a union of countably many closed subsets of $B$. By noetherian approximation (see 
	\cite[IV, \S 8.10]{ega4} or	\cite[Appendix C]{thomason1990higher}), the $G$-action on $\mc{X}$ descends to a subfield $k_0$ of $k$ which is finitely generated over its prime field. In other words, there exist 
		
		\begin{itemize}
		\smallskip \item
		a field $k_0\subset k$ finitely generated over its prime field, 
		
		\smallskip \item
		a smooth group scheme $G_0$ of finite type over $k_0$, 
		
		\smallskip \item
		$k_0$-schemes of finite type $B_0$ and $\mc{X}_0$, 
		
		\smallskip \item
		a $G_0$-action on $\mc{X}_0$ over $k_0$, 
		
		\smallskip \item
		a flat separated $G_0$-invariant morphism $f_0 \colon \mc{X}_0\to B_0$, and 
		
		\smallskip \item
		a Cartesian diagram 
		\begin{equation}\label{spread}
			\begin{tikzcd}
				\mc{X} \arrow[d,"f"] \arrow[r] & \mc{X}_0 \arrow[d,"f_0"]  \\
				B \arrow[r,"\pi"] & B_0,  
			\end{tikzcd}
		\end{equation}\noindent
		such that $G=G_0\times_{k_0}k$, and the base change of the $G_0$-action on $\mc{X}_0/B_0$ along $\pi$ is isomorphic to the $G$-action on $\mc{X}/B$.  \end{itemize}
		
		By \Cref{phi}(b), we have $\Phi_f(n)=\pi^{-1}(\Phi_{f_0}(n))$. Thus, since $\pi$ is continuous, it suffices to prove that $\Phi_{f_0}(n)$ is a countable union of closed subsets of $B_0$. In other words, we may assume that $k$ is finitely generated over its prime field and that $B$ is of finite type over $k$. In this case, $B$ is countable, hence $\Phi_{f}(n)$ is countable. It remains to show that $\Phi_{f}(n)$ is a union of closed subsets of $B$.  By elementary topology it suffices to show that $\Phi_f(n)$ is closed under specialization; see
		\cite[Tag 0EES]{stacks-project}. In other words, if $b' \in B$ is a specialization of $b \in \Phi_f(n)$, i.e., $b'\in\cl{\set{b}}$, then
		we want to show that $b' \in \Phi_f(n)$.
		
		By \cite[Proposition 7.1.4]{ega2}, there exist a discrete valuation ring $R$ with closed point $s$ and generic point $\eta$, and a morphism $\Spec R\to B$ sending $s$ to $b'$ and $\eta$ to $b$. Pre-composing with the completion map $\Spec \hat{R}\to \Spec R$, we may assume that $R$ is complete. Since $B$ is a $k$-scheme, the residue fields of $b,b',s,\eta$ all have the same characteristic as $k$. 
		Thus $R$ is complete and equicharacteristic and hence, by Cohen's Structure Theorem we have an isomorphism $R\simeq k(s)[[t]]$. 
		In particular, the residue field $k(s)$ is contained in $R$. By \Cref{ed-dvr}(b), letting $\cl{\eta}$ and $\cl{s}$ be geometric 
		points of $\Spec R$ lying above $\eta$ and $s$, respectively, we deduce that 
		\[\ed_{k(\cl{\eta})} \big( X_{k(\cl{\eta})} \big) \geqslant 
		\ed_{k(\cl{s})} \big( X_{k(\cl{s})} \big).\] 
		Now \Cref{phi}(a) tells us that
		\[ n \geqslant \ed_{{k(\cl{b})}} \big( X_{k(\cl{b})} \big) \geqslant \ed_{k(\cl{b}')} \big( X_{k(\cl{b}')} \big),\]
		where $\cl{b}$ and $\cl{b}'$ are geometric points of $B$ lying above $b$ and $b'$, respectively. This shows that $\Phi_f(n)$ is closed under specialization. 
		
		Assume now that $k$ is algebraically closed and of infinite transcendence degree over its prime field, and let $m$ be the maximum of $\ed_{k(\cl{b})}(\mc{X}_{\cl{b}};G_{k(\cl{b})})$, where $\cl{b}$ ranges over all geometric points of $B$. Consider the diagram (\ref{spread}). Since $\Phi_{f_0}(m-1)$ is a union of closed subsets of $B_0$ and it does not equal $B_0$, it does not contain the generic point of $B_0$. By \Cref{phi}(b), we have $\Phi_f(m-1)=\pi^{-1}(\Phi_{f_0}(m-1))$, hence for every $k$-point $b$ of $B$ mapping to the generic point of $B_0$ we have  $\ed_k(\mc{X}_b)=m$. By \Cref{mumford}, the set of such $k$-points is Zariski dense in $B$.
	\end{proof} 
	
\begin{rmk}\label{verygen-coh-inv}    To put \Cref{verygen} in perspective, 
     we will conclude this section by recalling an analogous 
	 result for cohomological invariants from~\cite[Appendix]{colliot2002exposant}.  
	 For an overview of the theory of cohomological invariants, see~\cite{garibaldi2003cohomological}.
		
	Let $k$ be a field, $G$ be a linear algebraic $k$-group, and $f \colon \mc{X} \to B$ be a morphism as 
	in~\Cref{verygen}. Let $i$ be a non-negative integer, $C$ be a finite $\on{Gal}(k_s/k)$-module of order prime to the characteristic of $k$, and $F\in \on{Inv}^i(G,C)$ be a cohomological invariant over $k$ with values in the Galois cohomology ring $H^i(-,C)$. 
	Passing to a dense open subscheme of $B$ if necessary, we may assume that $F(k(X))$ comes from a cohomology class $\alpha\in H^i_{\textrm{\'et}}(X,C)$. In this case, by the compatibility of the specialization map in \'etale and Galois cohomology 
	\cite[Page 15, Footnote]{garibaldi2003cohomological}, this implies that $\alpha_{\cl{s}}=F(k(X_{\cl{s}}))$ (up to sign) 
	for every geometric point $\cl{s}$ of $B$. From~\cite[Proposition A7]{colliot2002exposant}, we deduce the following:		
\[	B_0:=\set{s\in B: F \big( k(X_{\cl{s}}) \big)=0 \text{ for some geometric point $\cl{s}$ above $s$}}  \]
is a countable union of closed subsets of $B$. 
Note that by the Rigidity Property for \'etale cohomology \cite[Corollary VI.2.6]{milne1980etale}, one may replace ``some" by ``every" 
in the definition of $B_0$, as in \Cref{phi}(a).
\end{rmk}

	\section{Counterexamples}
	\label{sect.counterexamples}
	
	\begin{example} \label{ex.general}
	The following example shows that in \Cref{verygen} we may not replace ``countable union'' by ``finite union''. In this example $l$ will denote an odd prime. We will assume that the base field $k = \mathbb C$ is the field of complex numbers and will write $\ed$ in place of $\ed_{\mathbb C}$.
	
		 Let $A$ be a complex abelian variety. Any choice of $v_1,\dots, v_r\in A[\ell]$ gives rise to a $(\Z/\ell \Z)^r$-action on $A$ via $(n_1, \ldots, n_r) \colon a \mapsto a + n_1 v_1 + \ldots + n_r v_r$. This action is free if and only if $v_1, v_2, \ldots, v_r$ are linearly independent over $\Z / \ell \Z$. When we view $A$ as a $(\Z/\ell \Z)^r$-variety via this action, we will denote it by $(A;v_1,\dots,v_r)$.
		 We will focus on the case, where $r = 2$ and $A = E \times E$ is the direct product of two copies of a complex elliptic curve $E$.
		 More specifically we will investigate how $\ed(E \times E; v_1, v_2)$ depends on the choice of $E$, $v_1$ and $v_2$.

         	Recall that the endomorphism ring of an elliptic curve over $\C$ is either $\Z$ or an order in an imaginary quadratic field extension of $\Q$, and that all such rings arise as endomorphism rings of a complex elliptic curve. By the Chinese Remainder Theorem, there exist  
         	infinitely many negative integers $d\equiv 2,3 \pmod 4$ such that $d$ is not a square modulo $\ell$. 
		
        \smallskip
		(i) Let $E$ be an elliptic curve over $\C$ such that $\on{End}(E)\simeq \Z[\sqrt{d}]$; see \cite[p. 426]{silverman2009arithmetic}. 
		 We claim that $\ed(E\times E; (q,0),(0,q))=1$ for any $q\in E(\C)[\ell]\setminus\set{0}$. 
		 
		 It is obvious from the definition
		that $\ed(E\times E: (q,0),(0,q)) \geqslant 1$, so we only need to show that $\ed(E\times E; (q,0),(0,q)) \leqslant 1$.
		Let $\phi\in \on{End}(E)$ be such that $\phi^2 \colon E \to E$ is multiplication by $d$. Since $\phi$ is an endomorphism, it restricts to a group homomorphism $E(\C)[\ell]\to E(\C)[\ell]$. Fixing a $(\Z/\ell \Z)$-basis of $E(\C)[\ell]\simeq (\Z/\ell \Z)^2$, $\phi$ corresponds to a matrix $A\in \on{GL}_2(\Z/\ell \Z)$. The matrix $A$ does not have any eigenvalues in $\Z/\ell \Z$. Indeed, if $Av=\lambda v$ for some non-zero $v\in (\Z/\ell \Z)^2$ and $\lambda\in \Z/\ell \Z$, then $dv=A^2v=\lambda^2v$, hence $d=\lambda^2$ in $\Z/\ell \Z$, which is impossible as $d$ is not a square modulo $\ell$. It follows that $q$ and $\phi(q)$ are linearly independent, and so form a basis of $E(\C)[\ell]$.
	     Now $\phi \colon (E;q)\to (E;\phi(q))$ is a $\Z/\ell \Z$-equivariant morphism and the composition
		\[(E\times E; (q,0),(0,q))\xrightarrow{(\on{id},\phi)} (E\times E; (q,0),(0,\phi(q)))\xrightarrow{+} (E;(q,\phi(q))) \] 
		is a $(\Z/\ell \Z)^2$-compression.  Thus $\ed(E\times E; (q,0),(0,q)) \leqslant 1$, as desired.
		
	    \smallskip
		(ii) Let $E$ be an elliptic curve such that $\on{End}(E)=\Z$, and let $q\in E(\C)[\ell]\setminus\set{0}$. We claim that $\ed(E\times E; (q,0),(0,q))=2$. 
		
		Indeed, assume the contrary. Then there exists a dominant $(\Z/ \ell \Z)^2$-equivariant rational map 
		\[ f \colon E\times E \dashrightarrow C, \]
		where $E \times E$ stands for the $(\Z/ \ell \Z)^2$-variety $(E \times E; (q, 0), (0, q))$ and $C$ is some curve on which $(\Z/\ell \Z)^2$ acts faithfully. We may assume that $C$ is smooth and projective. Since $\ell$ is odd, $(\Z/ \ell \Z)^2$ cannot act faithfully on $\P^1$. Thus $C$ is not isomorphic to $\P^1$. For all but finitely many $v$, $f$ restricts to a well-defined surjective morphism $E\simeq E\times\set{v}\to C$. We deduce from Hurwitz's formula that $C$ has genus $1$. After suitably choosing an origin for $C$, $f$ becomes an everywhere defined homomorphism of abelian varieties. The restrictions of $f$ to $E\times \set{0}$ and $\set{0}\times E$ give isogenies $f_1,f_2 \colon E\to C$ such that
		th element $(1, 0)$ of $(\Z/ \ell \Z)^2$ acts on $C$ via translation by $f_1(q)$, and the element $(0, 1)$
		acts on $C$ via translations by $f_2(q)$. Since the $(\Z / \ell \Z)^2$-action on $C$ is faithful, we conclude that
		$f_1(q)$ and $f_2(q)$ form a basis of $C[\ell]$. On the other hand, recall from~\cite[Lemma 4.2(b)]{silverman2009arithmetic}
		that $\on{Hom}(E, C)$ is torsion-free $\Z$-module. Since
		\[ \on{Hom}(E,C)\otimes_{\Z}\Q\simeq \on{Hom}(E,E)\otimes_{\Z}\Q\simeq \Q , \]
		we conclude that $\on{Hom}(E,C)= \Z$. This implies that there exists homomorphism $h \colon E\to C$ such that $f_1$ and $f_2$ are multiples of $h$. In particular, $f_1(q)$ and $f_2(q)$ are linearly dependent, a contradiction. We conclude that $C$ does not exist, and 
		thus $\ed(E\times E; (q,0),(0,q))=2$, as claimed.
	
	    \smallskip	
		For every prime $\ell$, there exists a complex curve $B$ and a family of elliptic curves $\mc{E}\to B$, together with a nowhere zero $\ell$-torsion section, such that every pair $(E;q)$ where $E$ is a complex elliptic curve and $q\in E(\C)[\ell]\setminus \set{0}$ arises as a fiber of $\mc{E}\to B$; 
		see~\cite[Proposition A4]{colliot2002exposant}. The group $\Z/\ell \Z$ acts freely on $\mc{E}$ over $B$ by translations by the $\ell$-torsion section, and so $(\Z/\ell \Z)^2$ acts freely on the self-product $\Phi:\mc{E}\times_B\mc{E}\to B$ by translation. The fibers of $\Phi$ are triples $(E\times E; (q,0), (0,q))$, where $q\in E(\C)[\ell]\setminus\set{0}$. There are infinitely many fibers as in (i), and the fibers as in (ii) form a very general subset. Thus, the above discussion shows that the fiber of $\Phi$ over a very general point of the curve $B$ has essential dimension $2$, while countably many fibers have essential dimension $1$. 
	\end{example}
	
	\begin{example}\label{flatness-necessary} The following example shows that the flatness assumption in \Cref{verygen} is necessary.
		
		Let $n$ be a positive integer, and let $k$ be an algebraically closed field of characteristic not dividing $n$. Consider the affine plane $\A^2_k=\Spec k[x,y]$ with coordinates $x,y$, let $X\subset \A^2$ be defined by the equation $x(y^n-1)=0$, let $B=\A^1_k=\Spec k[x]$, and let $f$ be the projection induced by the inclusion $k[x]\subset k[x,y]$. The group $\mu_n=\Z/n\Z$ acts on $\A^2_k$ by $\zeta\cdot (x,y)\mapsto (x, \zeta y)$. Then $X$ is $\mu_n$-invariant, $f$ is $\mu_n$-equivariant, and the $\mu_n$-action on the fibers of $f$ is generically free and primitive. We have $\ed_k(X_a)=0$ for every $a\in k^{\times}$, but $\ed_k(X_0)=\ed_k(\mu_n)=1$.
	\end{example}
	
	\begin{example} \label{alg.closure-necessary}
		The following example shows that \Cref{verygen} fails if $k$ is not algebraically closed. More precisely, in this case the $k$-points
		$s \in B(k)$, where $\ed_k(X_s) \leqslant n$ do not necessarily lie on a countable union of closed subvarieties of $B$.
		
		Indeed, let $k = \mathbb R$ be the field of real numbers and $G$ be the orthogonal group $\Orth_2$ defined over $\mathbb R$.
		Consider the action of $G = \Orth_2$ on $X = \GL_2$ via multiplication on the right. Note that $X$ is the total space of a
		$G$-torsor $\tau \colon X \to Y$, where $Y = \GL_2/O_2$ is naturally identified with the space of symmetric $2 \times 2$ matrices via
		$\tau \colon A \mapsto A A^T$.
		
		Now consider the morphism 
		\[ f \colon X = \GL_2 \longrightarrow B = \mathbb A^1 \setminus \{ 0 \}, \] 
		sending a matrix $A$ to $\det(A)^2$. This morphism factors through $\tau$ as follows: 
		\[ f \colon X \stackrel{\tau}{\longrightarrow} Y \stackrel{\det}{\longrightarrow} B = \mathbb A^1 \setminus \{ 0 \}. \]
		Denote that fibers of $X$ and $Y$ over $s \in B$ by $X_s$ and $Y_s$, respectively. Then $X_s$ is a $G$-torsor over $Y_s$. 
		
		\smallskip
		{\bf Claim:} View a non-zero real number $s$ as an $\bbR$-point of $B$. Then 
			\begin{equation} \label{e.real} \ed(X_s) = \begin{cases} \text{$0$, if $s < 0$, and} \\ 
				\text{$1$, if $s > 0$.}
			\end{cases} \end{equation}
		
	    Note that $Y_s$ is the variety of symmetric matrices $B =\begin{pmatrix} a & b \\ b & c \end{pmatrix}$ such that $\det(B) = s$. 
	    Thus $Y_s$ is a rational surface over $\bbR$ whose function field can be identified with $\bbR(a, b)$. Passing to the generic point of
	    $Y_s$, we see that $\ed_{\mathbb R}(X_s) = \ed_{\mathbb R}(\tau_s)$, where $\tau_s \in H^1(\bbR(a, b), \Orth_2)$ 
	    is the $\Orth_2$-torsor over $\bbR(a, b)$ obtained by pulling back $\tau$ to the generic point of $Y_s$. Examining the long exact cohomology sequence associated to the exact sequence $1 \to \Orth_2 \to \GL_2$ of algebraic groups and remembering that $H^1(\mathbb R(a, b), \GL_2) = 1$
	    by Hilbert's Theorem 90, we see that 
	    $H^1(\bbR(a, b), \Orth_2)$ is in a natural bijective correspondence with the set of 2-dimensional non-singular quadratic forms 
	    over $\bbR(a, b)$, up to equivalence, and the quadratic form $q_s$ corresponding to $\tau_s$ is the form whose Gram matrix is 
	    $\begin{pmatrix} a & b \\ b & c \end{pmatrix}$, where $c = \dfrac{s + b^2}{a}$. Note that, by definition, $\ed_{\bbR}(\tau_s) = \ed_{\bbR}(q_s)$ and the discriminant of $q_s$ is $s$.
	    
	    Since $q$ assumes the value $a$ and has discriminant $s$, $q_s$ is isomorphic to $\ang{ a, \, as}$, 
	    Here $\ang{a, as}$ denote the $2$-dimensional quadratic form $q_s(z, w) = a z^2 + as w^2$ over
	    $\mathbb R(a, b)$. If $s < 0$, then $q$ is isotropic over $\bbR(a, b)$. Hence, $q$ is hyperbolic over $\bbR(a, b)$, i.e., $q_s$ is
		isomorphic to $\ang{ 1, -1}$; see~\cite[Theorem I.3.2]{lam2005introduction}. In particular, $q_s$ descends to $\mathbb R$ and hence, 
		$\ed_{\mathbb R}(q_s) = 0$.
		
		On the other hand, suppose that $s > 0$. Then $s$ is a complete square in $\mathbb R(a, b)$, so 
		$q \simeq \ang{ a, \, a}$. Clearly $q_s$ descends to $\mathbb R(a) \subset K$, so 
		$\ed_{\mathbb R}(q_s) \leqslant 1$. In order to complete the proof of~\eqref{e.real},
		it remains to show that $\ed_{\mathbb R}(q_s) \neq 0$. We argue 
		by contradiction. Assume $\ed_{\mathbb R}(q_s) = 0$, i.e., $q_s$ descends to some intermediate extension $\mathbb R \subset K \subset \bbR(a, b)$,
		where $\trdeg_{\mathbb R}(K) = 0$. In other words, $K$ is algebraic over $\mathbb R$. Since $\mathbb R$ is algebraically 
		closed in $\bbR(a, b)$, this is only possible if $K = \mathbb R$, i.e., $q$ descends to a $2$-dimensional form $q_0$ defined over $\mathbb R$.
		Since $s > 0$, $q$ is anisotropic over $\bbR(a, b)$, and hence, so is $q_0$.
		Let $\nu_a \colon K^{\times} \to \mathbb Z$ the valuation associated to the variable $a$. It is now
		easy to see that for any $(0, 0) \neq (f, g) \in (K^{\times})^2$,	$\nu_a(q_0(f, g))$ is even, 
		where as $\nu_a(q(f, g))$ is odd. This tells us that $q$ and $q_0$ have no values in common, contradicting our assumption that
		$q$ descends to $q_0$. Thus completes the proof of~\eqref{e.real}. \qed
	\end{example}

\section{Transversal intersections in projective space}

This section contains several preliminary results which will be used in the proof of \Cref{complete-int}. The common theme
is transversal intersections of projective varieties with linear subspaces in projective space. Note that there are no algebraic 
groups or group actions here; they will come into play in the next section.

	Recall that a commutative ring with identity is said to be regular if it is noetherian and all its localizations at prime ideals are regular local rings.
	
	\begin{lemma}\label{simultaneous-regular}
		Let $A$ be a regular semi-local noetherian ring and $\mathfrak{m}_1,\mathfrak{m}_2,\dots,\mathfrak{m}_r$ be the maximal ideals of $A$. For each $1\leqslant i\leqslant r$, let $P_i\subset \mathfrak{m}_i$ be a prime ideal such that $P_i\not\subset \mathfrak{m}_j$ for any $j\neq i$ and such that each local ring $A/P_i$ is regular. Assume that the prime ideals $P_1, \ldots, P_r$ have the same height,
		$\on{ht}(P_1) = \dots = \on{ht}(P_r) = c$. Then there exist $h_1, h_2,\dots, h_c \in A$ such that $P_iA_{\mathfrak{m}_i}=(h_1,\dots,h_c)A_{\mathfrak{m}_i}$ and  $h_1,\dots, h_c$ 
		form a regular sequence in $A_{m_i}$ for each $i$.
	\end{lemma}
	
	\begin{proof}
		For any $a \neq b$, the ideal $P_{a}+P_{b}$ is not contained in any $\mathfrak{m}_i$. Hence $P_{a}+P_{b}=A$. By the Chinese Remainder Theorem the natural ring homomorphism $A/(P_1^2\cdots P_r^2)\to (A/P_1^2)\times\dots\times (A/P_r^2)$ is an isomorphism; see~\cite[Tag 00DT]{stacks-project} or~\cite[Exercise 2.6]{eisenbud}.
		In particular, the natural projection $A\to (A/P_1^2)\times\dots\times (A/P_r^2)$ is surjective. Since $A/P_i$ is regular
		for every $1\leqslant i\leqslant r$, there exist $h_{i,1}, \dots, h_{i,c}\in P_i$ whose images $\cl{f}_{i,1},\dots,\cl{f}_{i,c}$ in the $A/\mathfrak{m}_i$-vector space $\mathfrak{m}_i/\mathfrak{m}_i^2$ form a basis of the subspace $P_i/\mathfrak{m}_i^2$, and such that $h_{i,j}-1\in P_s^2$ for $s \neq i$. In particular $h_{i,1},\dots,h_{i,c}\in P_i\setminus (\cup_{s \neq i} P_s)$. By Nakayama's Lemma $P_iA_{\mathfrak{m}_i}=(h_{i,1},\dots, h_{i,c})A_{\mathfrak{m}_i}$.

		Set $h_j:=\prod_{i=1}^n h_{i,j}$ for each $j = 1, \ldots, r$. Then $h_j-h_{i,j}\in P_i^2$ for all $i$, hence $h_j\in \cap_{i=1}^r P_i$ for all $j$. 
		Moreover, the images of $h_1,\ldots,h_c$ in $\mathfrak{m}_i/\mathfrak{m}_i^2$ form a basis of $P_i/\mathfrak{m}_i^2$ for all $i$. By \cite[Tag 00SC]{stacks-project},  ${h_1},\ldots, {h_c}$ form a regular sequence in $\mathfrak{m}_iA_{\mathfrak{m}_i}$ for each $i$. Since $\on{ht}(P_i)=c$, the $h_1,\ldots ,h_c$ generate $P_iA_{\mathfrak{m}_i}$ for each $i = 1, \ldots, r$.
	\end{proof}

	Let $k$ be a field. For every $0\leqslant d\leqslant n$, we denote by $\on{Gr}(n,n-d)$ the Grassmannian of codimension $d$ hyperplanes of $\P^n_k$. If $W \subset \P^n_k$ is a $k$-subspace of codimension $d$, we will
	denote by $[W]\in \on{Gr}(n,n-d)(k)$ the $k$-point representing $W$ in the Grassmannian. We will say that $[W]$ intersects $Z$ transversely at a smooth point $z \in Z$ if $z \in W$ and the tangent space $T_z(Z)$ intersects $W$ transversely. 
	Equivalently, $[W]$ intersects $Z$ transversely at a smooth point $z \in Z$ 
	if $W$ can be cut out by linear forms $h_1, \ldots, h_d \in \Gamma(\P^n_k,\mc{O}(1))$ such that
	$h_1(z) = \ldots, h_d(z) = 0$ and $\displaystyle h_1/h, \ldots, h_d/h$ form a regular sequence in the local ring $\mc{O}_{Z, z}$
	for some (and thus any) $h \in \Gamma(\P^n_k,\mc{O}(1))$ with $h(z) \neq 0$; see~\cite[Section 10.3]{eisenbud}.
	
	\begin{lemma}\label{grass-flat}
		Let $Z$ be a closed subscheme of $\P^n_k$, 
		$0 \leq c \leq n$ be an integer, 
		\[ I_{Z, c} \subset \on{Gr}(n, n-c)\times Z \] 
		be the incidence correspondence parametrizing pairs $([W],v)$ such that $v \in W$, and		
		\begin{align*} \phi \colon & I_{Z, c} \; \to \; \on{Gr}(n,n-c) \\
                                & ([V], v) \to V
        \end{align*} 
        be the projection to the first component. Assume that $z$ is a smooth $k$-point of $Z$ and a codimension $c$
        linear subspace $W_0$ of $\P^n$ intersects $Z$ transversely at $z$. Then $\phi$ is smooth at $([W_0], z)$.
	\end{lemma}
	
	
	\begin{proof} We claim that $\phi$ is flat at $([W_0], z)$. Note that 
	the fiber $W \cap Z$ of $\phi$ over $[W]$ is smooth at $z$, so the lemma follows from 
	this claim by~\cite[Tag 01V8]{stacks-project}.
	
	To prove the claim, we argue by induction on $c$. In the base case, $c = 0$,
	$\on{Gr}(n, n -c)$ is a point, $\phi$ is the identity map, and the claim is obvious.
	
	For the induction step, assume that $c \geqslant 1$ and the claim holds when $c$ is
	replaced by $c -1$, for every $n \geqslant c$ and every closed subscheme 
	$Z$ of $\P^n$. Choose linear forms $h_1, \ldots, h_c \in \Gamma(\P^n_k,\mc{O}(1))$ such that
	$h_1, \ldots, h_c$ cut out $W_0$, and $\displaystyle h_1/h, \ldots, h_c/h$ 
	form a regular sequence in the local ring $\mc{O}_{Z, z}$ for some  
	$h \in \Gamma(\P^n_k,\mc{O}(1))$ such that $h(z) \neq 0$.
	
	Denote the zero locus of $h_1$ by $\P^{n-1}$, the intersection $Z \cap \mathbb P^{n-1}$ by $Z'$, 
	the preimage of $Z'$ under $\phi$ by $I_{Z, c}'$, and the restriction of $\phi$ to $I_{Z, c}$ by $\phi'$.
	By~\cite[Corollary 6.9]{eisenbud} it suffices to show that $\phi' \colon I_{Z, c}' \to \Gr(n, n-c)'$ 
	is flat at $([W_0], z)$. Here $\Gr(n, n-c)'$ denotes the hypersurface in $\Gr(n, n-c)$ consisting of $(n-c)$-dimensional linear subspaces of
	$\mathbb P^n$ which are contained in $\mathbb P^{n-1}$.	We will view $Z'$ as a closed subscheme 
	of $\mathbb P^{n-1}$. Since $h_1$ cuts $Z$ transversely at $z$, $z$ is a smooth point of $Z'$. 
		Now observe that
	$\Gr(n, n-c)'$ is naturally isomorphic to $\Gr(n - 1, n - c)$ and $I_{Z, c}'$ is naturally isomorphic to $I_{Z', c-1}$ over $Z'$
	so that the following diagram commutes
		\[
		\begin{tikzcd}
	I_{Z', c-1} \arrow[d] \arrow[r, "\simeq"]	&	I'_{Z, c} \arrow[hook,r]  \arrow[d, "\phi'"] & I_{Z, c} \arrow[d,"\phi "] \\ 
	\Gr(n-1, n - c) \arrow[r, "\simeq"] &		\Gr(n, n-c)' \arrow[hook, r] & \Gr(n, n - c)
		\end{tikzcd}
		\]	
	Since $\Gr(n-1, n-c) = \Gr(n-1, \, (n-1) - (c-1))$, we can use the induction assumption to conclude that
	$\phi'$ is flat at $([W_0], z)$, as desired. This completes the induction step.
\end{proof}

	\begin{lemma}\label{grass2}
		Let $Z$ be an irreducible quasi-projective $k$-variety,  $Y\subset Z$ be a closed equidimensional subvariety, with irreducible components $Y_1,\dots,Y_r$. For every $1\leqslant i\leqslant r$, let $z_i\in Y$ be a closed point such that $Y$ and $Z$ are smooth at $z_i$, and $c$ be the codimension of $Y_i$ in $Z$. Then
		there exist an integer $n \geqslant 0$, a closed embedding $Z \hookrightarrow \P^n_k$ and a codimension $c$ subspace $W_0 \subset \P^n_k$ such that
		$W_0$ intersects $Z$ transversely at $z_i$, and locally around $z_i$ we have $Y= Y_i = Z\cap W_0$ (scheme-theoretically) for each $i = 1, \ldots r$.
%
%
%
	\end{lemma}

	\begin{proof}
		Since $Z$ is quasi-projective, there exists an affine open subset $\Spec B\subset Z$ containing $z_1,\dots,z_r$. Let $\Spec A\subset Z$ be the semi-localization of $Z$ at $\set{z_1,\dots,z_r}$. By definition, the ring $A$ is obtained from $B$ by localizing at the multiplicative subset consisting of elements which do not belong to the ideals corresponding to the $z_i$. Then $A$ is a regular semi-local ring. By \Cref{simultaneous-regular}, there exist $f_1,\dots,f_c\in A$, such that $f_1,\dots, f_c$ form a regular sequence in $\mc{O}_{Z,z_i}$ and generates the ideal of $Y_i$ in $\mc{O}_{Z,z_i}$, for each $1\leqslant i\leqslant r$. 
		
		 Since $Z$ is a quasi-projective, there exists a locally closed embedding $\iota: Z\hookrightarrow \P^n_k$. For every $0\leqslant i\leqslant c$, $f_i$ is the restriction of a rational function $P_i/Q_i$ on $\P^n_k$, where $P_i$ and $Q_i$ are homogeneous polynomials and $Q_i(z_h)\neq 0$. (Here, by $Q_i(z_h)$ we mean the image of $Q_i$ in the residue field of $k(z_h)$, which is a finite extension of $k$.)  We deduce that around every $z_h$ the variety $Y$ is the scheme-theoretic intersection of $Z$ and the closed subscheme defined by $P_1,\dots,P_c$. Multiplying the $P_i$ by some homogeneous polynomials not vanishing at the $z_h$, we may assume that the $P_i$ have the same degree $D\geqslant 1$. Replacing $m$ by $mD$ (which amounts to composing $\iota$ with a suitable Veronese embedding of $\P^n_k$), we may assume that the $P_i$ have degree $1$, that is, that $\set{P_i=0}\subset\P^n$ are hyperplanes. Now the embedding $\iota \colon Z \hookrightarrow \P^n_k$ and the linear subspace $W_0$ of $\P^n_k$ given by $\set{P_1=\dots=P_c=0}$ have the properties claimed 
		 in the lemma. 
%
%
	\end{proof}
	
	\section{Proof of Theorem~\ref{complete-int}}
	\label{sect.family}

    Let $X_0 = X_0^{(1)} \cup X_0^{(2)} \cup \ldots $ be the irreducible decomposition of $X_0$.
	Choose rational functions $\alpha_1, \ldots, \alpha_{d-1} \colon X_0 \dashrightarrow \A^1_k$ such that the restriction of
	$\alpha_1, \ldots, \alpha_{d-1}$ to $X_0^{(i)}$ generate the function field $k(X_0^{(i)})$ for every $i$.
	After adjoining all $G$-translates of $\alpha_1, \ldots, \alpha_{d-1}$ to this set, we may assume that $G$ permutes $\alpha_1, \ldots, \alpha_{d-1}$. 
	Consider the $G$-equivariant rational map
	$\alpha \colon X_0 \dasharrow \mathbb P(V)$ taking $x \in X_0$ to $(1 : \alpha_1(x): \ldots : \alpha_d(x))$ for a suitable linear (permutation)
	representation of $G$ on $V = k^{d + 1}$. Note that since the $G$-action on $X_0$ is assumed to be faithful,
	the $G$-action on $\P(V)$ is faithful as well. Moreover,
	by our construction, $\alpha$ induces a $G$-equivariant
	birational isomorphism between $X_0$ and $\alpha(X_0)$. After replacing $X_0$ with the closure of $\alpha(X_0)$ in $\mathbb P(V)$, we may assume that $X_0$ is a closed subvariety of $\P(V)$.
	
Let $\P(V)_{\rm non-free}$ be  the non-free locus for the $G$-action on $V$, i.e., the union of the fixed point loci $\P(V)^g$ as $g$ ranges over the non-trivial elements of $G$. Since $G$ acts faithfully on $\P(V)$, we have
\begin{equation} \label{e.non-free0} \dim \, \P(V) > \dim \, \P(V)_{\rm non-free}.
\end{equation}
Let $V^r$ be the direct sum of $r$ copies of $V$ (as a $G$-representation). We claim that
\begin{equation} \label{e.non-free}
\text{the codimension of $\P(V^r)_{\rm non-free}$ in $\P(V^r)$ is $\geqslant r$.}
\end{equation}
Indeed, for each $1 \neq g \in G$, let $\on{mult}(g, V)_{\lambda}$ be the dimension of the $\lambda$-eigenspace of $g$ and
	$\on{mult}(g, V)$ be the maximal value of $\on{mult}(g, V)_{\lambda}$, where $\lambda$ ranges over $\overline{k}$.
	Then $\dim \, \P(V)_{\rm non-free}$ is the maximal value of $\on{mult}(g, V) - 1$ as $g$ ranges over the non-trivial 
	elements of $G$. Clearly $\on{mult}(g, V^r) = r \on{mult}(g, V)$ for every $g$. Thus 
	\[ \dim \, \P(V^r)_{\rm non-free} = (\dim \, \P(V)_{\rm non-free} + 1) r - 1,  \]
	and the codimension of $\P(V^r)_{\rm non-free}$ in $\P(V^r)$ is 
	\begin{align*} \dim \, \P(V^r) - \dim \P(V^r)_{\rm non-free} & = r (\dim \, \P(V) + 1) - 1 - \big( (\dim \, \P(V)_{non-free} + 1) r - 1
	\big) \\	& = 
	r (\dim \, \P(V) - \dim \, \P(V)_{\rm non-free}) \geqslant r , \end{align*} 
where the last inequality follows from~\eqref{e.non-free0}.
This completes the proof of~\eqref{e.non-free}.

The $G$-equivariant linear embedding $V\hookrightarrow  V^r$, given by $v\mapsto (v,0,\dots,0)$, induces a $G$-equivariant closed embedding
\[X_0 \subset \P(V)\hookrightarrow \P(V^r). \]  This allows us to view $X_0$ as a $G$-invariant subvariety of $\P(V^r)$. 
By~\cite[Theorem 4.14]{popov1994} there exists a geometric quotient map $\pi \colon \P(V^r) \to \P(V^r)/G$. Explicitly, 
write \[\P(V^r)=\on{Proj}(k[t_0,\dots,t_{dr}]),\] where each variable $t_i$ has degree $1$, and let $A:=k[t_0,\dots,t_{dr}]^G$.
Then  $\P(V^r)/G\simeq \on{Proj}(A)$ and $\pi$ is induced by the inclusion $A^G \hookrightarrow A$ of graded rings.
Restricting $\pi$ to $X_0 \subset \P(V^r)$, we obtain the geometric quotient map $X_0 \to X_0/G$. Note that
$\dim \P(V^r)/G  = \dim \P(V^r) = dr - 1$ and $\dim \, X_0/G = \dim \, X_0 = e$, so every irreducible component of $X_0/G$ 
is of codimension $c = dr - 1 - e$ in $\P(V^r)/G$.
				
		We can choose smooth closed points $x_1,\dots,x_r$,	one on each irreducible component of $X_0$, such that 
		 the (scheme-theoretic) stabilizer $G_{x_i}$ is trivial for every $i$. 
		We now apply \Cref{grass2} to $Z=\P(V^r)/G$, $Y= X_0/G$, $z_1, \ldots, z_r$, 
		where $z_i = \pi(x_i)$ for each $i$. We deduce that there exist a closed 
		embedding $\P(V^r)/G \hookrightarrow \P^n$ defined over $k$ and a subspace $W_0 \subset \P^n$ 
		of codimension $c$ such that $X_0/G = W_0 \cap (\P(V^r)/G)$ locally around $z_i$ for each $i$. 
		Consider the diagram
		\[
		\begin{tikzcd}
			T \arrow[dr] \arrow[hook,rr] \arrow[bend right = 50,swap, "\cl{f}"]{ddrr} &  &  \Gr(n, n-c) \times \P(V^r) \arrow[d,"\pi\times \on{id}"]  \\
			& I_{Z, c} \arrow[hook,r]  \arrow[dr, swap, "\phi"] & \Gr(n, n-c) \times Z \arrow[d,"\on{pr}_2"] \\ 
		& 	& \Gr(n, n - c)
		\end{tikzcd}
		\]
		\noindent
		Here  $I_{Z, c}$ is the incidence correspondence parametrizing pairs $([W], q)$, where $W$ is a linear subspace of $\P^n_k$ of codimension $c$ and $q\in W$, as in~\Cref{grass-flat}, and $T$ is the preimage of $I_{Z, e}$ in $\P(V^r) \times \Gr(n, n-c)$. 
		
		We claim that the map $\phi$ is smooth at $([W_0], z_i)$ for each $i = 1, \ldots, r$. To prove this claim, note that flatness is local with respect to the fpqc topology; see~\cite[02L2]{stacks-project}. Since the morphism $\Spec \cl{k}\to \Spec k$ is fpqc, we may base-change 
		from $k$ to $\cl{k}$. By~\Cref{grass-flat}, $\phi_{\cl{k}}$ is smooth; hence, so is $\phi$. This proves the claim. (Alternatively, one could replace the fpqc topology by the \'etale topology as follows. The points $x_i\in X_0$ may be chosen so that the finite extensions $k(x_i)/k$ are separable. Since flatness is \'etale-local, to check flatness of $\phi$ at $([W_0], z_i)$ we may first base change to the Galois closure of $k(z_i)$, and then apply \Cref{grass-flat}. This yields an alternative proof of the claim.) By our choice of $x_1, \ldots, x_r$, $\pi$ is smooth 
		at each of these points; hence,  $\cl{f} = \phi \circ (\pi \times \text{id}) \colon T \to \Gr(n, n-c)$ is smooth  at $([W], x_i)$ for each $i$.
				
		We will construct the family $f \colon \mc{X} \to B$ by restricting $\cl{f}$ to a dense open 
		subset $\mc{X} = T \setminus C$, where $C = T_{\rm sing} \cup T_{\rm non-free}$. Here  
		$T_{\rm sing}$ is the singular locus of $f$ and $T_{\rm non-free} = \P(V^r)_{\rm non-free} \times \Gr(n, n-c)$ 
		is the non-free locus for the $G$-action in $T$. (Recall that $\P(V^r)_{\rm non-free}$ was defined
		 at the beginning of this section.)
	    The base $B$ of our family will be obtained by removing from $\on{Gr}(n,n-c)$ the locus of points $b \in \on{Gr}(n,n-c)$ such that
	    the entire fiber $T_b$ lies in $C$. In particular, $\cl{f}(\mc{X})\subset B$. Since $C$ is closed in $T$, $B$ is open in $\Gr(n, n-c)$.
	    Note also that since $\cl{f}$ is a proper morphism, $\cl{f}(C)$ is closed in $B$. 
	    
	    Let $b_0 := [W_0] \in \Gr(n, n-c)(k)$. By our choice of $x_1, \ldots, x_r$, 
	    none of the points $([W_0], x_i)$  lie in $C$. Hence $b_0 \in B(k)$ and the union of the irreducible components of $\mc{X}_{b_0}$
	    passing through $x_1, \ldots, x_r$ remains birationally isomorphic to $X_0$. Moreover, by our construction
	    $\mc{X}$ and $B$ are irreducible, the $G$-action on $\mc{X}$ is free, $f$ is smooth  of constant relative dimension $e = \dim(X_0) = \dim(\mc{X}_{b_0})$. In particular, conditions (i) and (ii) of Theorem~\ref{complete-int} hold for $f$.
	    It remains to check that for $r > e$, 
	    
	    \smallskip
	    (a) $\cl{f}(T_{\rm non-free}) \neq \Gr(n, n-c)$,
	    
	    \smallskip
	    (b) $\cl{f}(T_{\rm sing}) \neq \Gr(n, n-c)$, 
	    
	    \smallskip
	    (c) there exists a dense open subset $U \subset \Gr(n, n-c)$ such that the fibers of $f$ over $U$
	    are projective and irreducible. 
	    
	    \smallskip
	    Together, (a), (b) and (c) this will ensure that condition (iii) of Theorem~\ref{complete-int} holds as well.
        
        To prove (a), recall that by~\eqref{e.non-free}, the codimension of $P(V^r)_{\rm non-free}$ in $P(V^r)$ is $\geqslant r > e$, i.e., $\dim \, \pi (\P(V^r)_{\rm non-free}) < rd - 1 - e = c$. Consequently, an $(n-c)$-dimensional linear subspace $W$ in $\P^n$ in general position will intersect $\pi(\P(V^r)_{\rm non-free}) = \P(V)_{\rm non-free}/G$ trivially. Consequently, $\cl{f}^{-1}(W) \cap T_{\rm non-free} = \emptyset$. This proves (a).
        
        	To prove (b), recall that the fiber of the morphism $\phi \colon I_{Z, c} \to \Gr(n, n-c)$ over $[W]$ is $W \cap Z$. By Bertini's Theorem, there exists
		a dense open subset $U \subset  \Gr(n, n - c)$ consisting of $(n - c)$-dimensional linear subspaces $W$ of $\P^n$ 
		such that $W \cap Z$ is smooth and irreducible. By generic flatness, after replacing $U$ by a smaller open subset, 
		we may assume that $\phi \colon \phi^{-1}(U) \to U$ is flat. Appealing to~\cite[Tag 01V8]{stacks-project}, as we did in~\Cref{grass-flat},
		we see that since $\phi \colon \phi^{-1}(U) \to U$ is a flat map with smooth fibers,
		it is smooth. Finally, after intersecting $U$ with the complement of $\cl{f}(T_{\rm non-free})$ (which we 
		know is a dense open subset of $\Gr(n, n-c)$ by (a)), we may assume that $\pi \times \on{id}$ is smooth over
		$\phi^{-1}(U)$. Hence, the map $\cl{f} \colon \cl{f}^{-1}(U) \to U$, being a composition of two smooth maps, $\pi \times \on{id} \colon \cl{f}^{-1}(U) \to \phi^{-1}(U)$ and $\phi \colon \phi^{-1}(U) \to U$, is smooth. This proves (b).
		
         Finally, to prove (c), recall that	by definition, for $b = [W] \in B$, the fiber $\phi^{-1}(b) = W \cap Z =
         W \cap \P(V^r)/G$ is a complete intersection in $\P(V^r)/G$. The fiber $\mc{X}_{[W']} = f^{-1}(W')$ is cut out by the same homogeneous polynomials, now viewed as elements of $k[V^r]$ instead of $k[V^r]^G$. Thus $\mc{X}_{b} $ is a smooth complete intersection in $\P(V^r)$. Since $\dim \, \mc{X}_b = e \geqslant 1$, by~\cite[n.~78]{serre1955faisceaux}, $\mc{X}_b$ is connected, hence irreducible. This completes the proof of (c) and thus of Theorem~\ref{complete-int}. 
	\qed
	
\begin{rmk} The family $f \colon \mc{X} \to B$ constructed in the proof of \Cref{complete-int} has the additional property that
the fibers of $f$ over the dense open subset $U \subset B$ (the open subset in part (iii) of the statement of the theorem) 
are complete intersections in projective space $\P^n$. With a bit of extra effort one can make sure that every 
fiber of $f$ over $U$ is of general type. Since we will not need this assertion in this paper, we leave the 
proof as an exercise for the interested reader. 
\end{rmk}
	
	\section{Proof of Theorem~\ref{thm.free}}
	\label{sect.thm.free}
	
	Let $d:=\dim(X)$ and $e:=\ed_k(X)$. 
	We must show that there exists an irreducible smooth projective variety $Z$ and a $G$-torsor $Y \to Z$ such that $\dim(Y) = \dim(X)$ and $\ed_k(Y) = \ed_k(X)$.
	If $e = 0$, Theorem~\ref{thm.free} is obvious: we can take $Y = G \times \mathbb P^d$, where $G$ acts by translations on the first factor and trivially on the second.
	Thus we may assume without loss of generality that $e \geqslant 1$ and in particular, $d \geqslant 1$. In this case, we use the following strategy: construct a family $\mc{X} \to B$ as in Theorem~\ref{complete-int} with $X_0 = X$, 
	then take $Y = \mc{X}_b$ where $b$ is a $k$-point of $B$ in very general position. 
	\Cref{complete-int} tells us that $\dim(Y) = d$ and the $G$-action on $Y$ is strongly unramified.
	 We would like to appeal to~\Cref{verygen} to conclude that $\ed_k(Y) = e$. One difficulty in implementing this strategy is that
	
\smallskip
	(i) \Cref{complete-int} requires $G$ to be a finite group, whereas in~\Cref{thm.free}, $G$ is an arbitrary linear algebraic group over a field of good characteristic.
	
\smallskip
\noindent
     Even if assume that $G$ is a finite group, there is another problem:

\smallskip     
     (ii) \Cref{verygen} requires the fibers of $f \colon \mc{X} \to B$ to be primitive $G$-varieties, whereas in \Cref{complete-int} the fiber $\mc{X}_{b_0}$ over $b_0$ may not be primitive; we only know that it contains the primitive variety $X_0 = X$ 
     as a union of its irreducible components.

\smallskip
\noindent
We will overcome (i) by using reduction of structure to a finite group, and (ii) by using the following variant of \Cref{verygen}.

	\begin{thm}\label{verygen-sing}
		Let $G$ be a linear algebraic group over an algebraically closed field
		$k$ of good characteristic (see Definition~\ref{assume}), $f:\mc{X}\to B$ be a $G$-equivariant morphism of $k$-varieties such that $B$ is irreducible, $G$ acts trivially on $B$, and that the generic fiber of $f$ is a primitive and generically free $G_{k(B)}$-variety. Let $b_0\in B(k)$, $x_0\in \mc{X}_{b_0}(k)$, and $X_0$ be a $G$-invariant reduced open subscheme of $\mc{X}_{b_0}$ containing $x_0$ such that:
		\begin{enumerate}
			\item $X_0$ is a generically free primitive $G$-variety, and
			\item $f$ is flat at $x_0$.
		\end{enumerate} 
		Then, for a very general $b\in B(k)$, $\mc{X}_b$ is generically free and primitive, and $\ed_k(\mc{X}_b)\geqslant \ed_k(X_0)$. Furthermore, if $k$ is of infinite transcendence degree over its prime field (in particular, if $k$ is uncountable), then set of those $b\in B(k)$ such that $\ed_k(\mc{X}_b)\geqslant \ed_k(X_0)$ is Zariski dense in $B$.
	\end{thm}
	
	Here, as usual, ``very general" means ``away from a countable union of proper subvarieties". 
	If $k$ has infinite transcendence degree over its prime field, 
	one may replace the assumption that $\mc{X}_{k(B)}\to \Spec k(B)$ is generically free and primitive with the assumption 
	that a general fiber of $f$ is generically free and primitive. 

	\begin{proof}
		Let $Z:=\mc{X}_{b_0}\setminus X_0$. Then $Z$ is a $G$-invariant closed subscheme of $X$, and the restriction $f|_{\mc{X}\setminus Z}:\mc{X}\setminus Z\to B$ is a flat $G$-equivariant morphism, whose fiber at $b_0$ is a generically free irreducible $G$-variety. Replacing $\mc{X}$ by $\mc{X}\setminus Z$, we may thus assume that $\mc{X}_{b_0}$ is a generically free primitive $G$-variety. 
		
		By generic flatness \cite[Tag 0529]{stacks-project}, there exists a dense open subscheme $U\subset B$ such that 
		$f$ is flat over $U$. If $b_0\in U(k)$, the conclusion follows from \Cref{verygen}. 
		We may thus assume that $b_0$ does not lie in $U$, and let $u\in U$ be the generic point of $U$. Since $U$ is dense in $B$, $b_0$ is a specialization of $u$. By \cite[Proposition 7.1.9]{ega2}, there exist a discrete valuation ring $R$ and a separated morphism $\Spec R\to B$ mapping the closed point $s$ of $\Spec R$ to $b_0$ and the generic point $\eta$ of $\Spec R$ to $u$.
		
		Since $f$ is flat at $x_0$ and $G$ acts primitively on on $\mc{X}_{b_0}$, by the openness of the flat locus \cite[Tag 0399]{stacks-project}, 
		the flat locus of the base change $f_R \colon \mc{X}_R\to \Spec R$ is dense in the component of the special fiber of $f_R$ containing 
		(the preimage of) $x_0$. Since we are assuming that the generic fiber is primitive, we conclude that the flat locus of $f_R$ is dense 
		in the generic fiber. Therefore, after removing the complement of the flat locus from $\mc{X}_R$, we may thus assume that $f_R$ is flat. 
		Let $\cl{\eta}$ and $\cl{s}$ be geometric points lying above $\eta$ and $s$, respectively. By \Cref{ed-dvr}(b), we have
		\[\ed_{k(\cl{\eta})}(\mc{X}_{\cl{\eta}})\geqslant \ed_{k(\cl{s})}((X_0)_{\cl{s}}).\] 
		On the other hand, by~\Cref{rigidity},
		\[ \ed_{k(\cl{s})}((X_0)_{\cl{s}})  \geqslant \ed_k(X_0).\] 
		Combining these two inequalities, we see that $\ed_{k(\cl{\eta})}(X_{\cl{\eta}})  \geqslant \ed_k(X_0)$. The conclusion now follows from an application of \Cref{verygen} to the restriction of $f$ to $f^{-1}(U)$. 
	\end{proof}	
	
	\begin{proof}[Conclusion of the proof of \Cref{thm.free}]
	

As we mentioned at the beginning of this section, in the course of proving \Cref{thm.free} we may assume that $e \geqslant 1$.

     We will now reduce the theorem to the case, where $X$ is incompressible, i.e., $d = e + \dim(G)$. Indeed, by the definition of essential dimension there exists a $G$-equivariant dominant rational map $X \dasharrow X'$ such that 
     $\dim(X') = d' := e + \dim(G)$.
     Suppose we know that \Cref{thm.free} holds for $X'$. In other words,
     there exists a $G$-variety $Y'$ such that $\dim(Y') = d'$, $\ed(Y') = e$, 
     and $Y'$ is the total space of a $G$-torsor $t \colon Y' \to P'$ over a smooth projective variety $P$. Clearly
     $d \geqslant d'$. Let $Y = \P^{d - d'} \times Y'$, where $G$ acts trivially. Then $\dim(Y) = d$, $\ed(Y) = \ed(Y') = e$ 
     (see, e.g., \cite[Lemma 3.3(d)]{reichstein2000} and $Y$ is a $G$-torsor
     ${\rm id} \times t \colon Y = \P^{d - d'} \times Y' \to \P^{d-d'} \times P$ over the smooth projective variety $\P^{d - d'} \times P$, as desired.
     
     From now on we will assume that $d  = e + \dim(G)$ and $e \geqslant 1$.
	 By \cite[Theorem 1.1]{chernousov2006resolving}, there exists a finite subgroup $S$ of $G$, such that for every field extension $K/k$ the natural map $H^1(K,S)\to H^1(K,G)$ is surjective. In particular, there exists a generically free primitive $S$-variety $X_0$ such that $X_0 \times^{S} G$ is birationally equivalent to $X$. Note that $\dim(X_0) = \dim(X) - \dim(G) = d - \dim(G)$. 
		
		Let $f \colon \mc{X} \to B$ be a family obtained by applying \Cref{complete-int} to the finite group $S$ and the $S$-variety $X_0$. This is possible because $S$ is a finite group, $X_0$ is primitive, hence equidimensional, and $e \geqslant 1$. Note also that since $k$ is algebraically closed,
		the assumption of \Cref{complete-int} that $k$-points should be dense in $X_0$ is automatically satisfied here.
		
		Recall that the family $f \colon \mc{X} \to B$ supplied to us by \Cref{complete-int} has the following properties:
		$f$ is smooth of constant relative dimension $e = \dim(X_0)$, 
		$S$ acts freely on $\mc{X}$, and there is a dense open subscheme $U \subset B$ 
		for any $b \in U(k)$ the fiber $\mc{X}_b$ is smooth, projective and irreducible.
		Moreover, $X_0$ is $S$-equivariantly birationally isomorphic to a union of 
		components of the fiber $\mc{X}_{b_0}$ for some $b_0 \in B(k)$. 
		
		Set $\mc{X}' := \mc{X} \times^{S} G$ and let 
		$f' \colon \mc{X}' \to B$ be the natural projection induced by $f$. 	
		Then $f'$ is a smooth morphism of constant relative dimension 
		$d = e + \dim(G)$, $G$ acts freely on $\mc{X}'$, $\mc{X}'_b$ is a primitive $G$-variety for 
		every $b \in U(k)$, the $G$-action on $\mc{X}'_b$ is strongly unramified, and $X$ 
		is $G$-equivariantly birationally isomorphic to a union of irreducible components of $\mc{X}'_{b_0}$.
		Clearly 
		\[ \ed(\mc{X}'_b) \leqslant \dim ( \mc{X}'_b) - \dim(G) = \dim(X) = d - \dim(G) = e . \]  
		On the other hand, since $k$ is algebraically closed and of infinite transcendence degree over its prime field, by \Cref{verygen-sing} there exists a $k$-rational point $b$ of $B$ such that
		$\ed_k(\mc{X}'_b) \geqslant \ed_k(X) = e$. 
		Setting $Y := \mc{X}'_b$, we obtain a generically free primitive $G$-variety $Y$ such that  $\dim(Y) = d$, $\ed_k(Y) = e$, and the $G$-action on $Y$
		is strongly unramified, as desired.
	\end{proof}
	
	\begin{cor} \label{thm.free2}
		Let $G$ be a linear algebraic group over an algebraically closed field $k$ of good characteristic and of infinite transcendence degree over its prime field.
		Then there exists an unramified generically free $G$-variety $Y$ such that
		\begin{enumerate}
			\item $\on{dim}(Y)=\ed_{k}(G)+\dim G$, and
			\item $\ed_k(Y)=\ed_k(G)$.
		\end{enumerate}	
	\end{cor}
	
    \begin{proof} Let $V$ be a generically free $G$-variety over $k$ of essential dimension $e = \ed_k(G)$.
    Then there is a $G$-compression $V \dasharrow X$, where $X$ is a generically free $G$-variety of dimension $e + \dim(G)$.
    Clearly $\ed_k(X) \leqslant e$; on the other hand, $\ed_k(X, G) \leqslant \ed_k(V, G) = e = \dim(X)$. We conclude
    that $\dim(X) = e + \dim(G)$ and $\ed_k(X) = e$. By Theorem~\ref{thm.free} there exists an unramified 
    generically free variety $Y$ of dimension $e + \dim(G)$ and essential dimension $e$.
    \end{proof}

\section{Finite group actions on hypersurfaces}

We conclude this section with yet another application of~\Cref{verygen-sing}.

	\begin{prop} \label{ex.hypersurface}
		Let $k$ be an algebraically closed field of infinite transcendence degree over its prime field, $G$ be a finite group, $W$ be a faithful $k$-representation of $G$, and set $V:=W\oplus k$, where $G$ acts trivially on $k$. Then for every $d\geqslant 1$, a very general polynomial $f\in (k[V]_d)^G$ cuts out an affine hypersurface $Z(f)$ of essential dimension $\ed_k(Z(f) ;G)=\ed_k(G)$.
	\end{prop}
	
Here as usual, $k[V]_d$ denotes the vector space of homogeneous polynomials of degree $d$ on $V$, and $(k[V]_d)^G$ denotes the subspace of $G$-invariant polynomials in $k[V]_d$.
	
	\begin{proof}
		Let $n:=\dim W$, let $x_1,\dots,x_n$ be coordinates on $W$, and let $x_{n+1}$ be a coordinate on $k$. We fix a polynomial function $f_0$ of degree $d$ on $W\oplus k$ which only depends on $x_{n+1}$, $f_0=f_0(x_{n+1})$, and which has distinct roots. Then $f_0\in (k[V]_d)^G\setminus\set{0}$ and $H_{f_0}$ is a disjoint union of $d$ translates of $W$.
		
		Let $B$ be the complement of the origin in $\A((k[V]_d)^G)$. Consider the commutative triangle
		\[
		\begin{tikzcd}
			\mc{X} \arrow[hook,r]  \arrow[dr] & \A(V)\times B \arrow[d,"\on{pr}_2"] \\
			& B,
		\end{tikzcd}
		\]\noindent
		where $\mc{X}$ is the reduced closed subscheme whose $k$-points parametrize pairs $(v,f)\in V\times k[V]^G_d$ such that $f \neq 0$ and $f(v)=0$. 
		By definition the fiber $\mc{X}_f$ over $f \in B$ is the hypersurface $Z(f)$ in the affine space $\mathbb A(V)$. In other words, 
		the morphism $\mc{X}\to B$ is the restriction to $B$ of the universal family of affine hypersurfaces in $V$ of degree $d$. In particular, it is flat.
		
			Every irreducible component of $\mc{X}_{[f_0]}$ is $G$-equivariantly isomorphic to $W$
		and so has essential dimension equal to $\ed_k(G)$; see~\cite[Proposition 3.11]{merkurjev2013essential}.
		Let $\mc{X}'\subset \mc{X}$ be the Zariski open subvariety obtained by removing all irreducible components of $\mc{X}_{[f_0]}$ but one.
		Since the inclusion map $\mc{X}' \hookrightarrow \mc{X}$ is an open embedding, and open embeddings are flat, we conclude that 
		the composition $\mc{X}' \hookrightarrow \mc{X} \to B$ is also flat. 
		Thus for a very general $f \in (k[V]_d)^G$, we have
		\[ \ed_k(Z(f)) = \ed_k(\mc{X}_{[f]}) = \ed_k(\mc{X}'_{[f]}) \geqslant \ed_k(\mc{X}'_{[f_0]}) = \ed_k(G) \]
		where the inequality follows from \Cref{verygen-sing}. The opposite inequality, $\ed_k(Z(f)) \leqslant \ed_k(G)$
		is immediate from the definition of $\ed_k(G)$.
	\end{proof}

    \section{Appendix: Essential dimension at a prime}

    The purpose of this appendix is to point out that \Cref{ed-dvr}, \Cref{verygen} and~\Cref{thm.free} continue to hold if we replace
    essential dimension with essential dimension at a prime $q$. The proofs are largely unchanged. One notable feature of these results
    is that they hold for any prime $q \neq \Char(k)$, where $k$ is a base field. \Cref{ed-dvr} and Theorems~\ref{thm.free} and \ref{verygen}
    assume that $k$ is a field of good characteristic; this assumption is not needed here. 

    	\begin{prop}\label{ed-dvr-prime}
	Let $k$ be an algebraically closed field, $R$ be a discrete valuation ring containing $k$ and with residue field $k$, and $l$ be the fraction field of $R$. Let $G$ be a linear algebraic group over $k$ and $q$ be a prime number invertible in $k$. Let $X$ be a flat separated $R$-scheme of finite type endowed with a $G$-action over $R$, whose fibers are generically free and primitive $G$-varieties. 
	Then $\ed_{\cl{l},q}(X_{\cl{l}})\geqslant \ed_{k,q}(X_{k})$.
	\end{prop}
	
	\begin{proof}
	The proof is the same as that of \Cref{ed-dvr}, except that instead of the~\cite[Theorem~1.2]{specialization1}, one should use 
	\cite[Theorem~11.1]{specialization1} which gives an analogous assertion for essential $q$-dimension.
	\end{proof}

    \begin{thm}\label{verygen-prime}
		Let $G$ be a linear algebraic group over an algebraically closed field $k$ and let $q$ be a prime number invertible in $k$. Let $B$ be a noetherian $k$-scheme, $f \colon \mc{X}\to B$ be a flat separated $G$-equivariant morphism of finite type such that $G$ acts trivially on $B$ and the geometric fibers of $f$ are generically free and primitive $G$-varieties (in particular, reduced). Then for any fixed integer $n \geqslant 0$ the subset of $b\in B$ such that $\ed_{k(\cl{b}),q}(\mc{X}_{\cl{b}};G_{k(\cl{b})})\leqslant n$ for every (equivalently, some) geometric point $\cl{b}$ above $b$ is a countable union of closed subsets of $B$.
    \end{thm}
	
    \begin{proof}
    The proof is analogous to that of \Cref{verygen}, replacing \Cref{ed-dvr} by \Cref{ed-dvr-prime}.
    \end{proof}

    \begin{thm}
    Let $G$ be a linear algebraic group over an algebraically closed field
		$k$ of infinite transcendence degree over its prime field, $q$ be a prime number invertible in $k$, and $X$ be a generically free primitive $G$-variety.
		Then there exists an irreducible smooth projective variety $Z$ and a $G$-torsor $Y \to Z$ such that $\dim(Y) = \dim(X)$ and $\ed_{k,q}(Y) = \ed_{k,q}(X)$.
    \end{thm}
    
    \begin{proof}
    Analogous to that of \Cref{thm.free}, replacing \Cref{verygen} by \Cref{verygen-prime}.
    \end{proof}

	\section*{Acknowledgements}
	We are grateful to Najmuddin Fakhruddin for helpful discussions. The second-named author thanks K{\fontencoding{T1}\selectfont\k{e}}stutis \v{C}esnavi\v{c}ius and the D\'epartement de Math\'ematiques d'Orsay (Universit\'e Paris-Saclay) for hospitality during Summer 2021, and the Institut des Hautes \'Etudes Scientifiques for hospitality in the Fall 2021.


\begin{thebibliography}{ABC99}
		
				\bibitem[Ana73]{sivaramakrishna1973schemas}
		S. Anantharaman.
		\newblock Sch\'emas en groupes, espaces homog\`enes et espaces alg\'ebriques
		sur une base de dimension 1.
		\newblock In {\em Sur les groupes alg\'ebriques}, number~33 in M\'emoires de la
		Soci\'et\'e Math\'ematique de France, pages 5--79. Soci\'et\'e math\'ematique
		de France, 1973.
		
		
		\bibitem[BRV07]{brosnan2007-arxiv}
		P. Brosnan, Z. Reichstein, and A. Vistoli.
		\newblock Essential dimension and algebraic stacks.
		\newblock {\em arXiv preprint arXiv:math/0701903}, 2007.
		
		
		\bibitem[BRV18]{brosnan2018essential}
		P. Brosnan, Z. Reichstein, and A. Vistoli.
		\newblock Essential dimension in mixed characteristic.
		\newblock {\em Doc. Math.}, 23:1587--1600, 2018.
		
		
		
		\bibitem[CGR06]{chernousov2006resolving}
		V. Chernousov, Ph. Gille, and Z. Reichstein.
		\newblock Resolving $G$-torsors by abelian base extensions.
		\newblock {\em Journal of Algebra}, 296(2):561--581, 2006.
		
		\bibitem[CGR08]{chernousov2008reduction}
		V. Chernousov, Ph. Gille, and Z. Reichstein.
		\newblock Reduction of structure for torsors over semilocal rings. 
		\newblock {\em Manuscripta Math.} 126 (2008), no. 4, 465-480.
		
		
		\bibitem[CT02]{colliot2002exposant}
		J.-L. Colliot-Th\'{e}l\`ene.
		\newblock Exposant et indice d'alg\`ebres simples centrales non ramifi\'{e}es.
		\newblock {\em Enseign. Math. (2)}, 48(1-2):127--146, 2002.
		\newblock With an appendix by Ofer Gabber.
		
		
		
		
		\bibitem[Eis95]{eisenbud}
		D. Eisenbud, 
		\newblock {\it Commutative algebra}.
		\newblock  Springer-Verlag, New York, 1995. MR1322960
		\newblock Graduate Texts in Mathematics, No. 150.
		
		\bibitem[FS21]{fakhruddin2021finite}
		N. Fakhruddin and R. Saini.
		\newblock Finite groups scheme actions and incompressibility of galois covers:
		beyond the ordinary case.
		\newblock {\em arXiv preprint arXiv:2102.05993}, 2021.
		
		\bibitem[FKW21a]{fkw}
		B. {Farb}, M. {Kisin}, and J. {Wolfson}.
		\newblock {The Essential Dimension of Congruence Covers}.
		\newblock {\em Compositio Mathematica}, 157(11), 2407--2432, 2021.
		
		\bibitem[FKW21b]{fkw2}
		B. {Farb}, M. {Kisin}, and J. {Wolfson}.
		\newblock {The Essential Dimension via Prismatic Cohomology}.
		\newblock {\em arXiv e-prints}, arXiv:2110.05534 , 2021.
		
		
		\bibitem[GMS03]{garibaldi2003cohomological}
		S. Garibaldi, A. Merkurjev, and J.-P. Serre.
		\newblock {\em Cohomological invariants in {G}alois cohomology}, volume~28 of
		{\em University Lecture Series}.
		\newblock American Mathematical Society, Providence, RI, 2003.
		
		
		\bibitem[GD61]{ega2}
		A. Grothendieck and J. Dieudonn{\'e}.
		\newblock {\em {\'E}l{\'e}ments de g{\'e}om{\'e}trie alg{\'e}brique {II}},
		volume~8 of {\em Publications {M}ath{\'e}matiques}.
		\newblock Institute des {H}autes {\'E}tudes {S}cientifiques., 1961.
		
		\bibitem[GD64]{ega4}
		A. Grothendieck and J. Dieudonn{\'e}.
		\newblock {\em {\'E}l{\'e}ments de g{\'e}om{\'e}trie alg{\'e}brique {IV}},
		volume 20, 24, 28, 32 of {\em Publications {M}ath{\'e}matiques}.
		\newblock Institute des {H}autes {\'E}tudes {S}cientifiques., 1964-1967.
		
		
		%
		
		\bibitem[Lam05]{lam2005introduction}
		T.~Y. Lam.
		\newblock {\em Introduction to quadratic forms over fields}, volume~67 of {\em
			Graduate Studies in Mathematics}.
		\newblock American Mathematical Society, Providence, RI, 2005.
		
		
		
		\bibitem[Liu02]{liu2002algebraic}
		Q. Liu.
		\newblock {\em Algebraic geometry and arithmetic curves}, volume~6 of {\em
			Oxford Graduate Texts in Mathematics}.
		\newblock Oxford University Press, Oxford, 2002.
		\newblock Translated from the French by Reinie Ern\'{e}, Oxford Science
		Publications.
		
		\bibitem[Mat89]{matsumura1989commutative}
		H. Matsumura.
		\newblock {\em Commutative ring theory}, volume~8 of {\em Cambridge Studies in
			Advanced Mathematics}.
		\newblock Cambridge University Press, Cambridge, second edition, 1989.
		\newblock Translated from the Japanese by M. Reid.
		
		\bibitem[Mer09]{merkurjev2009essential}
		A.~S. Merkurjev.
		\newblock Essential dimension.
		\newblock In {\em Quadratic forms---algebra, arithmetic, and geometry}, volume
		493 of {\em Contemp. Math.}, pages 299--325. Amer. Math. Soc., Providence,
		RI, 2009.
		
		\bibitem[Mer13]{merkurjev2013essential}
		A.~S. Merkurjev.
		\newblock Essential dimension: a survey.
		\newblock {\em Transformation groups}, 18(2):415--481, 2013.
		
		\bibitem[Mil80]{milne1980etale}
		J.~S. Milne.
		\newblock {\em \'{E}tale cohomology}, volume~33 of {\em Princeton Mathematical
			Series}.
		\newblock Princeton University Press, Princeton, N.J., 1980.
		
		
		\bibitem[PV94]{popov1994}
		V.~L. Popov and E.~B. Vinberg.
		\newblock {\em Invariant Theory}, pages 123--278.
		\newblock Springer Berlin Heidelberg, Berlin, Heidelberg, 1994.
		
		
		\bibitem[Rei00]{reichstein2000}
		Z. Reichstein, 
		\newblock 
		On the notion of essential dimension for algebraic groups.
		\newblock
		{\em Transform. Groups}, vol. 5 (3) (2000), 265--304. 
		
		
		
		
		\bibitem[RS21]{specialization1}
		Z. Reichstein and F. Scavia.
		\newblock The behavior of essential dimension under specialization.
		\newblock 
		
		\bibitem[Ser55]{serre1955faisceaux}
		J.-P. Serre.
		\newblock Faisceaux alg\'{e}briques coh\'{e}rents.
		\newblock {\em Ann. of Math. (2)}, 61:197--278, 1955.
		
		\bibitem[Ser58]{serre58charp}
		J.-P. Serre.
		\newblock Sur la topologie des vari\'{e}t\'{e}s alg\'{e}briques en
		caract\'{e}ristique {$p$}.
		\newblock In {\em Symposium internacional de topolog\'{\i}a algebraica
			{I}nternational symposium on algebraic topology}, pages 24--53. Universidad
		Nacional Aut\'{o}noma de M\'{e}xico and UNESCO, Mexico City, 1958.
		
		\bibitem[Ser79]{serre1979local}
		J.-P. Serre.
		\newblock {\em Local fields}.
		\newblock Translated from the French by Marvin Jay Greenberg. Graduate Texts in Mathematics, 67. Springer-Verlag, New York-Berlin, 1979. viii+241 pp. ISBN: 0-387-90424-7 
		
		
		
		\bibitem[Sil09]{silverman2009arithmetic}
		J.~H. Silverman.
		\newblock {\em The arithmetic of elliptic curves}, volume 106 of {\em Graduate
			Texts in Mathematics}.
		\newblock Springer, Dordrecht, second edition, 2009.
		
		\bibitem[SP]{stacks-project}
		The {Stacks Project Authors}.
		\newblock \textit{Stacks Project}.
		\newblock http://stacks.math.columbia.edu.
		
		\bibitem[TT90]{thomason1990higher}
		R.~W. Thomason and T. Trobaugh.
		\newblock Higher algebraic {$K$}-theory of schemes and of derived categories.
		\newblock In {\em The {G}rothendieck {F}estschrift, {V}ol.\ {III}}, volume~88
		of {\em Progr. Math.}, pages 247--435. Birkh\"auser Boston, Boston, MA, 1990.
		
		
		
	\end{thebibliography}
\end{document}